\documentclass[10pt,reqno]{amsart}
\usepackage{amsmath}
\usepackage{amssymb}
\usepackage{color}

\newtheorem{theorem}{Theorem}[section]
\newtheorem{remark}{Remark}[section]
\newtheorem{lemma}[theorem]{Lemma}
\newtheorem{proposition}[theorem]{Proposition}

\newtheorem{define}{Definition}[section]

\begin{document}
\title[SIRS-B epidemic PDE model]{Global stability and uniform persistence of the reaction-convection-diffusion cholera epidemic model}
\author{Kazuo Yamazaki}
\address{Washington State University, Department of Mathematics and Statistics, Pullman, WA 99164-3113, U.S.A. (kyamazaki@math.wsu.edu)}
\author{Xueying Wang}
\address{Washington State University, Department of Mathematics and Statistics, Pullman, WA 99164-3113, U.S.A. (xueying@math.wsu.edu)}  
\date{}
\maketitle

\begin{abstract}
We study the global stability issue of the reaction-convection-diffusion cholera epidemic PDE model and show that the basic reproduction number serves as a threshold parameter that predicts whether cholera will persist or become globally extinct. Specifically, when the basic reproduction number is beneath one, we show that the disease-free-equilibrium is globally attractive. On the other hand, when the basic reproduction number exceeds one, if the infectious hosts or the concentration of bacteria in the contaminated water are not initially identically zero, we prove the uniform persistence result and that there exists at least one positive steady state. 

\vspace{5mm}

\end{abstract}
\footnote{2010MSC: 35B65, 35K57, 47H20, 92D25, 92D30}

\section{Introduction}
Cholera is an ancient intestinal disease for humans. It has a renowned place in epidemiology with John Snow's famous investigations of London cholera in 1850's which established the link between contaminated water and cholera outbreak. Cholera is caused by bacterium {\it vibrio cholerae}. The disease transmission consists of two routes:  indirect environment-to-human  (through ingesting the contaminated water) and direct person-to-person transmission routes.  Even though cholera has been an object of intense study for over a hundred years, it remains to be a major public health concern in developing world;  the disease has resulted in a number of outbreaks including the recent devastating outbreaks in Zimbabwe and Haiti, and renders more than 1.4 million cases of infection and 28,000 deaths worldwide every year \cite{WHO2}. 

It is well known that the transmission and spread of infectious diseases are complicated by spatial variation that involves distinctions in ecological and geographical environments, population sizes, socio-economic and demographic structures, human activity levels, contact and mixing patterns, and many other factors.  In particular, for cholera, spatial movements of humans and water can play an important role in shaping complex disease dynamics and patterns \cite{DB11,PBFHPGMR11}. There have been many studies published in recent years on cholera modeling and analysis (see, e.g., \cite{BCGR09,CP2,C14,HMS06,MLWGSM11,NSGFL10,SD11,TW11,TE10,WL12,WM11, WPW16,WW15, YW16}).  However, only a few mathematical models 
among this large body of cholera models have considered human and water movement so far. Specifically, Bertuzzo et al. incorporated both water and human movement and formulated a simple PDE model \cite{BCGR09, RBMRBGCMVR12} and a patch model \cite{BMRGCBRR11}, in which only considered indirect transmission route. Chao et al. \cite{CHL11}  proposed a stochastic model to study vaccination strategies and accessed its impact on spatial cholera outbreak in Haiti by using the model and data, for which both direct and indirect transmission were included. Tien, van den Driessche and their collaborators used network ODE models  incorporating both water and human movement between geographic regions, and their results establish the connection in disease threshold between network and regions \cite{ESTD13,TSED15}.  Wang et al. \cite{WPW16} developed a generalized PDE model to study the spatial spread of cholera dynamics along a theoretical river, employing general incidence functions for direct and indirect transmission and intrinsic bacterial growth and incorporating both human/pathogen diffusion and bacterial convection. 

In the present paper, we shall pay our attention to a  reaction-diffusion-convection cholera model, which employs a most general formulation incorporating all different factors. This PDE model was first proposed in \cite{WPW16} and received investigations \cite{WPW16, YW16}. Let us now describe this model explicitly in the following section. 

\section{Statement of Main Results}
We study the following SIRS-B epidemic PDE model for cholera dynamics with $x \in [0,1], t > 0$: 
\begin{subequations}
\label{Eqn:PDE-model}	
\begin{align}\label{1}
& \frac{\partial S}{\partial t} = D_{1} \frac{\partial^{2} S}{\partial x^{2}} +b - \beta_{1}SI - \beta_{2}S\frac{B}{B+K}   -d S+\sigma R,\\
&\frac{\partial I}{\partial t} = D_{2} \frac{\partial^{2} I}{\partial x^{2}} +  \beta_{1}SI + \beta_{2}S\frac{B}{B+K} - I(d + \gamma),\\
&\frac{\partial R}{\partial t} = D_{3} \frac{\partial^{2} R}{\partial x^{2}}  + \gamma I- R(d + \sigma),\\
&\frac{\partial B}{\partial t} = D_{4} \frac{\partial^{2} B}{\partial x^{2}} - U\frac{\partial B}{\partial x} + \xi I + gB\left(1-\frac{B}{K_{B}}\right)-\delta B,
\end{align}
\end{subequations}
(cf. \cite{WPW16}) subjected to the following initial and Neumann and Robin boundary conditions respectively:  
\begin{equation}\label{2}
S(x,0) = \phi_{1}(x), \hspace{3mm} 
I(x,0) = \phi_{2}(x), \hspace{3mm} 
R(x,0) = \phi_{3}(x), \hspace{3mm} 
B(x,0) = \phi_{4}(x), \hspace{3mm} 
\end{equation}
where each $\phi_{i} (i = 1, 2, 3, 4)$ is assumed to be nonnegative and continuous in space $x$, and 
\begin{subequations}\label{3}
\begin{align}
& \frac{\partial Z}{\partial x} (x, t)\Big|_{x = 0, 1} = 0,\hspace{2cm} Z=S, I, R,\\
& D_{4}\frac{\partial B}{\partial x} (x, t) - UB(x,t)\rvert_{x = 0} = \frac{\partial B}{\partial x}(x,t) \rvert_{x = 1} = 0.
\end{align}
\end{subequations}
Here $S=S(x,t), I=I(x,t), $ and $R=R(x,t)$ measure the number of susceptible, infectious, and recovered human hosts at location $x$ and time $t$, respectively. $B=B(x,t)$ denotes the concentration of the bacteria (vibrios) in the water environment. 
The definition of model parameters is provided in Table \ref{tab1}.

\begin{table}[h]
\caption{Definition of parameters in model \eqref{Eqn:PDE-model}}. \label{tab1}
\vspace{1mm}
\begin{tabular}{ ll}
\hline
\hline
Parameter  & Definition\\
\hline
$b$& Recruitment rate of susceptible hosts \\
$d$ & Natural death rate of human hosts \\
$\gamma$ &  Recovery rate of infectious hosts \\
$\sigma$& Rate of host immunity loss \\
$\delta$ & Natural death rate of bacteria\\
$\xi$ & Shedding rate of bacteria by infectious hosts\\
$\beta_{1}$ & Direct transmission parameter\\
$\beta_{2}$ & Indirect transmission parameter\\
$K$ & Half saturation rate of bacteria\\
$U$ & Bacterial convection coefficient\\
$K_{B}$ & Maximal carrying capacity of bacteria in the environment\\
\hline
\hline
\end{tabular}
\end{table}

We assume all of these parameters to be positive. Hereafter let us write $\partial_{t}, \partial_{x}, \partial_{xx}^{2}$ for $\frac{\partial}{\partial t}, \frac{\partial}{\partial x}, \frac{\partial^{2}}{\partial {x}^{2}}$, respectively.

To state our results clearly, let us denote the solution 
\begin{equation}\label{4}
u = (u_{1}, u_{2}, u_{3}, u_{4}) \triangleq (S, I, R, B) \in \mathbb{R}^{4}, \hspace{5mm} \phi \triangleq (\phi_{1}, \phi_{2}, \phi_{3}, \phi_{4}).
\end{equation}
We also denote the Lebesgue spaces $L^{p}$ with their norms by $\lVert \cdot\rVert_{L^{p}},\,p \in [1, \infty]$. Finally, we denote  
\begin{equation}\label{5}
X \triangleq C([0,1], \mathbb{R}^{4}) = \prod_{i=1}^{4} X_{i}, \hspace{5mm} X_{i} \triangleq C([0,1], \mathbb{R}),
\end{equation}
the space of $\mathbb{R}^{4}$-valued functions continuous in $x \in [0,1]$ with the usual sup norm 
\begin{equation}\label{6}
\lVert u\rVert_{C([0,1])} \triangleq \lVert S \rVert_{C([0,1])} + \lVert I \rVert_{C([0,1])} + \lVert R \rVert_{C([0,1])} + \lVert B \rVert_{C([0,1])}. 
\end{equation}
We define analogously 
\begin{equation*}
X^{+} \triangleq C([0,1], \mathbb{R}_{+}^{4}) = \prod_{i=1}^{4} X_{i}^{+}, \hspace{3mm} X_{i}^{+} \triangleq \{f \in C([0,1], \mathbb{R}): f \geq 0\}. 
\end{equation*}
	
Understanding the global dynamical behavior of cholera modeling problems is crucial in order to suggest effective measures to control the growth of the disease. To the best of our knowledge, the existing literature has only studied local dynamics of solutions of this general PDE model. The focus of the present work is global disease threshold dynamics, which will be established in terms of the basic reproduction number $R_0$ \cite{HWZ13, T09,WZ11}.  To that end, we conduct a rigorous investigation on the disease using the model, and analyze both model parameters and the system dynamics for a better understanding of disease mechanisms. Particularly, we perform a careful analysis on the global threshold dynamics of the disease. 

In review of previous results, firstly the authors in \cite{WW15} defined $\mathcal{R}_{0}^{ODE}$ for the SIRS-B ODE model, which can be extended to the SIRS-B PDE model as follows: denoting   
\begin{equation}\label{7}
\Theta_{1} \triangleq 
\begin{pmatrix}
m^{\ast} \beta_{1} & m^{\ast} \frac{\beta_{2}}{K}\\
\xi & g 
\end{pmatrix}, 
\hspace{5mm} 
\Theta_{2} \triangleq 
\begin{pmatrix}
D_{2} \partial_{xx}^{2} - (d+\gamma) & 0 \\
0 & D_{4} \partial_{xx}^{2} - U \partial_{x} - \delta 
\end{pmatrix},
\end{equation}
where $m^{\ast} \triangleq \frac{b}{d}$, we have $\mathcal{R}_{0}^{PDE} \triangleq r(-\Theta_{1} \Theta_{2}^{-1})$, the spectral radius of $-\Theta_{1}\Theta_{2}^{-1}$, for which $\mathcal{R}_{0}^{ODE}$ is same except that the operators $\Theta_{1}, \Theta_{2}$ in (\ref{7}) would have no diffusive operators $\partial_{xx}^{2}$. Moreover, the authors in \cite{WW15} proved that when $\mathcal{R}_{0}^{ODE} \leq 1$, the model has the disease-free-equilibrium (DFE) $(S, I, R, B) = (m^{\ast}, 0, 0, 0)$ which is globally asymptotically stable (see Theorem 2.1 of \cite{WW15}). On the other hand, when $\mathcal{R}_{0}^{ODE} > 1$, it was proven that this ODE model has two equilibriums, namely the DFE which is unstable and endemic equilibrium which is globally asymptotically stable (see Theorem 2.1 \cite{WW15}). For the SIRS-B PDE model with diffusion, the authors in \cite{YW16} used spectral analysis tools from \cite{T09} to show that when $\mathcal{R}_{0}^{PDE} < 1$, the DFE is locally asymptotically stable while if $\mathcal{R}_{0}^{PDE} > 1$, then there exists $\eta > 0$ such that any positive solution of (\ref{1}) linearized at the DFE satisfies 
\begin{equation}\label{8}
\limsup_{t\to \infty} \lVert (S(\cdot, t), I(\cdot, t), R(\cdot, t), B(\cdot, t)) - (m^{\ast}, 0, 0, 0) \rVert_{C([0,1])} \geq \eta.
\end{equation}
We emphasize here that both these stability and persistence results were local; specifically the results were obtained via analysis on the $(S, I, R, B)$ that solves the system (\ref{1}) linearized at the DFE $(m^{\ast}, 0, 0, 0)$, not necessary the actual system (\ref{1}). The major difficulty was that because by definition $\mathcal{R}_{0}^{PDE}$ gives information only on the linearized system (see the definition $\mathcal{R}_{0}^{PDE} = r(-\Theta_{1}\Theta_{2}^{-1})$, (\ref{7}), (\ref{13}), (\ref{14})), it seemed difficult to utilize the hypothesis that $\mathcal{R}_{0}^{PDE}  >1$ or $\mathcal{R}_{0}^{PDE} < 1$ to deduce any information on the actual system (\ref{1}) (see e.g. Theorem 4.3 (ii) of \cite{WZ12}).  

In this paper, we overcome this major obstacle and extend these stability results to global; moreover, we obtain the uniform persistence result. We also extend Lemma 1 of \cite{LZ11}, which have proven to be useful in various other papers (e.g. Lemma 3.2, \cite{VWZ12}) to the case with convection, which we believe will be useful in many future work. For simplicity, let us hereafter denote $\mathcal{R}_{0} \triangleq \mathcal{R}_{0}^{PDE}$, and by $u(x,t,\phi)$ the solution at $(x,t) \in [0,1] \times [0, \infty)$ that initiated from $\phi$: 

\begin{theorem}
Suppose $D = D_{1} = D_{2} = D_{3}, \phi \in X^{+}$. Then the system (\ref{1}) subjected to (\ref{2}), (\ref{3}) admits a unique global nonnegative solution $u(x,t,\phi)$ such that $u(x, 0, \phi) = \phi(x)$. Moreover, if $\mathcal{R}_{0} < 1$, then the DFE $(m^{\ast}, 0, 0, 0)$ is globally attractive. 
\end{theorem} 

\begin{theorem}
Suppose $D = D_{1} = D_{2} = D_{3}, \phi \in X^{+}$ and $g < \delta$. Let $u(x,t,\phi)$ be the unique global nonnegative solution of the system (\ref{1}) subjected to (\ref{2}), (\ref{3}) such that $u(x, 0, \phi) = \phi(x)$ and $\Phi_{t}(\phi) = u(t, \phi)$ be its solution semiflow. If $\mathcal{R}_{0} > 1$ and $\phi_{2}(\cdot) \not\equiv 0$ or $\phi_{4}(\cdot) \not\equiv 0$, then the system (\ref{1}) admits at least one positive steady state $a_{0}$ and there exists $\eta > 0$ such that
\begin{align}\label{9}
\liminf_{t\to\infty} u_{i}(x,t) \geq \eta, \hspace{3mm} \forall i = 1, 2, 4, 
\end{align}
uniformly $\forall \hspace{1mm} x \in [0,1]$.
\end{theorem}

\begin{remark}
\begin{enumerate}
\item []	
\item We remark that typically the persistence results in the case $\mathcal{R}_{0} > 1$ requires a hypothesis that the solution is positive (see e.g. Theorem 4.3 (ii) of \cite{WZ12} and also Theorem 2.3 (2) of \cite{YW16}). In the statement of Theorem 2.2, we only require that $\phi_{2}(\cdot) \not\equiv 0$ or $\phi_{4}(\cdot) \not\equiv 0$. Due to the Proposition 6.1, we are able to relax these conditions. Moreover, we note that $\sup$ in (\ref{8}) is replaced by $\inf$ in (\ref{9}). 
\item The proof was inspired by the work of \cite{LZ11, VWZ12, WZ11}. 
\item We remark that it remains unknown what happens when $\mathcal{R}_{0} = 1$; for this matter, not global but even in the local case, it remains an open problem (see Theorem 2.3 \cite{YW16}). 
\item In the system (\ref{1}), we chose a particular case of 
\begin{equation*}
f_{1}(I) = \beta_{1}I, \hspace{3mm} f_{2}(B) = \beta_{2}\frac{B}{B+K}, \hspace{3mm} h(B) = gB \left(1 - \frac{B}{K_{B}}\right)
\end{equation*}
where $f_{1},f_{2}, h$ represent the direct, indirect transmission rates, intrinsic growth rate of bacteria respectively (see \cite{WW15, WPW16}). We remark for the purpose of our subsequent proof that defining this way, $f_{1}, f_{2}, h$ are all Lipschitz. It is clear from the proof that some generalization is possible. 
\end{enumerate}
\end{remark}

The rest of the article is organized as follows. The next section presents preliminary results of this study. Section 4 verifies a key proposition as an extension of Lemma 1 of \cite{LZ11}, which has proved to be useful in various context.  Our main results are established in Sections 5-6. By employing the theory of monotone dynamical systems \cite{Z03}, we prove that (1) the disease free equilibrium (DFE) is globally asymptotically stable if the basic reproduction number $R_0$ is less than unity; (2) there exists at least one positive steady state and the disease is uniformly persistent in both the human and bacterial populations if $R_0>1$. Additionally, we identify a precise condition on model parameters for which the system admits a unique nonnegative solution, and study the global attractivity of this solution. In the end, a brief discussion is given in Section 7, followed by Appendix. 

\section{Preliminaries}
When there exists a constant $c= c(a,b) \geq 0$ such that $A \leq cB, A = cB$, we write $A \lesssim_{a,b} B, A \approx_{a,b} B$. 
 
Following \cite{S95, YW16}, we let $A_{i}^{0}, i = 1, 2, 3$ denote the differentiation operator 
\begin{align*}
A_{i}^{0}u_{i} \triangleq D \partial_{xx}^{2} u_{i}, \hspace{3mm} A_{4}^{0} \triangleq D_{4} \partial_{xx}^{2} u_{4} - U \partial_{x} u_{4}, 
\end{align*}
defined on their domains 
\begin{align*}
& D(A_{i}^{0}) \triangleq \{\psi \in C^{2}((0,1)) \cap C^{1}([0,1]): A_{i}^{0} \psi \in C([0,1]), \partial_{x} \psi \rvert_{x=0,1} = 0 \},  \,\, i=1,2,3,\\
&D(A_{4}^{0}) \triangleq \{\psi \in C^{2}((0,1)) \cap C^{1}([0,1]):\\
& \hspace{33mm} A_{4}^{0} \psi\in C([0,1]), D_{4} \partial_{x} \psi - U \psi \rvert_{x=0} = \partial_{x} \psi \rvert_{x=1} = 0 \},
\end{align*}
respectively. We can then define $A_{i}, (i = 1, 2, 3, 4)$ to be the closure of $A_{i}^{0}$ so that $A_{i}$ on $X_{i}$ generates an analytic semigroup of bounded linear operator $T_{i}(t), t \geq 0$ such that $u_{i}(x,t) = (T_{i}(t) \phi_{i})(x)$ satisfies 
\begin{equation*}
\partial_{t} u_{i}(t) = A_{i} u_{i}(t), \hspace{3mm} u_{i}(0) = \phi_{i} \in D(A_{i}) 
\end{equation*}
where 
\begin{equation*}
D(A_{i}) = \left\{\psi \in X_{i} : \lim_{t\to 0^{+}} \frac{(T_{i}(t) - I)\psi}{t} = A_{i} \psi \text{ exists } \right\};
\end{equation*}
that is, for $i = 1, 2, 3$, 
\begin{equation*}
\partial_{t} u_{i}(x,t) = D_{i} \partial_{xx}^{2} u_{i}(x,t), t > 0, x \in (0,1), \hspace{3mm} 
\partial_{x} u_{i} \rvert_{x= 0 ,1} = 0, \hspace{3mm} u_{i}(x,0) = \phi_{i}(x), 
\end{equation*}
and 
\begin{equation*}
\begin{cases}
\partial_{t} u_{4}(x,t) = D_{4} \partial_{xx}^{2} u_{4}(x,t) - U \partial_{x} u_{4} (x,t), \hspace{3mm} t > 0, x \in (0,1),\\
D_{4} \partial_{x} u_{4} - U u_{4} \rvert_{x=0} = \partial_{x} u_{4} \rvert_{x=1} = 0, \hspace{3mm} u_{4}(x,0) = \phi_{4}(x). 
\end{cases}
\end{equation*}
It follows that each $T_{i}$ is compact (see e.g. pg. 121 \cite{S95}). Moreover, by Corollary 7.2.3, pg. 124 \cite{S95}, because $X_{i}^{+} = C([0,1], \mathbb{R}_{+})$, each $T_{i}(t)$ is strongly positive (see Definition \ref{pg. 38, 40, 46, [Z03]}). 

We now let 
\begin{subequations}\label{10}
\begin{align}
& F_{1} \triangleq b - \beta_{1} SI - \beta_{2} S \left(\frac{B}{B+K}\right) - dS + \sigma R,\\
& F_{2} \triangleq \beta_{1} SI + \beta_{2} S \left(\frac{B}{B+K}\right) - I(d+\gamma),\\
& F_{3} \triangleq \gamma I - R(d+\sigma),\\
& F_{4} \triangleq \xi I + gB \left(1 - \frac{B}{K_{B}}\right) - \delta B, 
\end{align}
\end{subequations}
and $F \triangleq (F_{1}, F_{2},F_{3}, F_{4})$. Let $T(t): X \mapsto X$ be defined by $T(t) \triangleq \prod_{i=1}^{4} T_{i}(t)$ so that it is a semigroup of operator on $X$ generated by $A \triangleq \prod_{i=1}^{4} A_{i}$ with domain $D(A) \triangleq \prod_{i=1}^{4} D(A_{i})$ and hence we can write (\ref{1}) as 
\begin{equation*}
\partial_{t} u = A u + F(u), \hspace{3mm} u(0) = u_{0} = \phi.
\end{equation*}
We recall some relevant definitions 
\begin{define}\label{pg. 2, 3, 11 [Z03]}
(pg. 2, 3, 11 \cite{Z03}) Let $(Y,d)$ be any metric space and $f: Y \mapsto Y$ a continuous map. A bounded set $A$ is said to attract a bounded set $B \subset Y$ if $\lim_{n\to\infty} \sup_{x\in B} d(f^{n}(x), A) = 0$. A subset $A \subset Y$ is an attractor for $f$ if $A$ is nonempty, compact and invariant $(f(A) = A)$, and $A$ attracts some open neighborhood of itself. A global attractor for $f$ is an attractor that attracts every point in $Y$. Moreover, $f$ is said to be point dissipative if there exists a bounded set $B_{0}$ in $Y$ such that $B_{0}$ attracts each point in $Y$. Finally, a nonempty invariant subset $M$ of $Y$ is isolated for $f: Y \mapsto Y$ if it is the maximal invariant set in some neighborhood of itself. 
\end{define}
 
\begin{define}\label{pg. 38, 40, 46, [Z03]}
(pg. 38, 40, 46, \cite{Z03}) Let $E$ be an ordered Banach space with positive cone $P$ such that int$(P) \neq \emptyset$. For $x, y \in E$, we write $x \geq y$ if $x- y \in P, x > y$ if $x-y \in P \setminus \{0\}$, and $x \gg y$ if $x-y \in \text{int}(P)$.  
	
A linear operator $L$ on $E$ is said to be positive if $L(P) \subset P$, strongly positive if $L(P\setminus \{0\}) \subset \text{int} (P)$. For any subset $U$ of $E$, $f: U \mapsto U$, a continuous map, $f$ is said to be monotone if $x \geq y$ implies $f(x) \geq f(y)$, strictly monotone if $x > y$ implies $f(x) > f(y)$, and strongly monotone if $x > y$ implies $f(x) \gg f(y)$. 
	
Let $U \subset P$ be nonempty, closed, and order convex. Then a continuous map $f: U \mapsto U$ is said to be subhomogeneous if $f(\lambda x) \geq \lambda f(x)$ for any $x \in U$ and $\lambda \in [0,1]$, strictly subhomogeneous if $f(\lambda x) > \lambda f(x)$ for any $x \in U$ with $x \gg 0$ and $\lambda \in (0,1)$, and strongly subhomogeneous if $f(\lambda x) \gg \lambda f(x)$ for any $x \in U$ with $x \gg 0$ and $\lambda \in (0,1)$.  
\end{define}

\begin{define}\label{pg. 56, 129 [S95]}
(pg. 56, 129, \cite{S95}) An $n\times n$ matrix $M = (M_{ij})$ is irreducible if $\forall  \hspace{1mm} I \subsetneq N = \{1, \hdots, n\}$, $I \neq \emptyset$, there exists $\hspace{1mm} i \in I$ and $j \in J = N \setminus I$ such that $M_{ij} \neq 0$. Moreover, $F: [0,1] \times \Lambda \mapsto \mathbb{R}^{n}, \Lambda$ any nonempty, closed, convex subset of $\mathbb{R}^{n}$, is cooperative if $\frac{\partial F_{i}}{\partial u_{j}}(x,u) \geq 0,  \hspace{1mm} \forall  \hspace{1mm} (x, u) \in [0,1]\times \Lambda, i \neq j$. 
\end{define}

\begin{lemma}\label{Theorem 7.3.1, Corollary 7.3.2, [S95]}
(Theorem 7.3.1, Corollary 7.3.2, \cite{S95}) Suppose that $F: [0,1] \times \mathbb{R}_{+}^{4} \mapsto \mathbb{R}^{4}$ has the property that 
\begin{equation*}
F_{i}(x, u) \geq 0 \hspace{3mm} \forall  \hspace{1mm} x \in [0,1], u \in \mathbb{R}_{+}^{4} \text{ and } u_{i} = 0. 
\end{equation*} 
Then $\forall \hspace{1mm} \psi \in X^{+}$, 
\begin{equation*}
\begin{cases}
\partial_{t}u_{i}(x,t) = D_{i} \partial_{xx}^{2}u_{i}(x,t) + F_{i}(x, u(x,t)), &t > 0, x \in (0,1),\\
\alpha_{i}(x) u_{i}(x,t) + \delta_{i} \partial_{x}u_{i}(x,t) = 0, &t > 0, x = 0, 1,\\
u_{i}(x,0) = \psi_{i}(x), &x \in (0,1),
\end{cases}
\end{equation*}
has a unique noncontinuable mild solution $u(x,t, \psi) \in X^{+}$ defined on $[0, \sigma)$ where $\sigma = \sigma(\psi) \leq \infty$ such that if $\sigma < \infty$, then $\lVert u(t) \rVert_{C([0,1])} \to \infty$ as $t \to \sigma$ from below. Moreover, 
\begin{enumerate}
\item $u$ is  continuously differentiable in time on $(0, \sigma)$, 
\item it is in fact a classical solution,
\item if $\sigma(\psi) = +\infty$ $\forall  \hspace{1mm} \psi \in X^{+}$, then $\Psi_{t}(\psi) = u(t, \psi)$ is a semiflow on $X^{+}$,  
\item if $Z\subset X^{+}$ is closed and bounded, $t_{0} > 0$ and $\cup_{t \in [0, t_{0}]} \Psi_{t}(Z)$ is bounded, then $\Psi_{t_{0}}(Z)$ has a compact closure in $X^{+}$. 
\end{enumerate}
\end{lemma}

\begin{remark}\label{pg. 121, [S95]}
This lemma remains valid even if the Laplacian is replaced by a general second order differentiation operator; in fact, all results from Chapter 7, \cite{S95} remain valid for a general second order differentiation operator (see pg. 121, \cite{S95}). In relevance we also refer readers to Theorem 1.1, \cite{MS90}, Corollary 8.1.3 \cite{W96} for similar general well-posedness results. 
\end{remark}

The following result was obtained in \cite{YW16}:
\begin{lemma}\label{Theorems 2.1, 2.2, [YW16]}
(Theorems 2.1, 2.2, \cite{YW16}) $\forall \hspace{1mm} \phi \in X^{+}$ the system (\ref{1}) subjected to (\ref{2}) and (\ref{3}) admits a unique nonnegative mild solution on the interval of existence $[0, \sigma)$ where $\sigma = \sigma(\phi)$. If $\sigma < \infty$, then $\lVert u(t) \rVert_{C([0,1])}$ becomes unbounded as $t$ approaches $\sigma$ from below. 

Moreover, if $D_{1} = D_{2} = D_{3}$, then $\sigma = + \infty$. Therefore, $\Phi_{t}(\phi) = u(t,\phi)$ is a semiflow on $X^{+}$. 
\end{lemma}
\begin{remark}
In the statement of Theorems 2.1, 2.2 of \cite{YW16}, we required the initial regularity to be in $X^{+}\cap H^{1}([0,1])$ where $H^{1}([0,1]) = \{f: f, \partial_{x} f \in L^{2}([0,1])\}$ and obtained higher regularity beyond $C([0,1], \mathbb{R}^{4})$; here we point out that to show the global existence of the solution $u(t) \in X^{+}  \hspace{1mm} \forall  \hspace{1mm} t \geq 0$, it suffices that the initial data is in $X^{+}$. For completeness, in the Appendix we describe the estimate more carefully than that of Proposition 1 in \cite{YW16} that is needed to verify this claim. 
\end{remark}

\begin{lemma}\label{Theorem 2.3.2, [Z03]}
(Theorem 2.3.2, \cite{Z03}) Let $E$ be an ordered Banach space with positive cone $P$ such that int$(P) \neq \emptyset$, $U \subset P$ be nonempty, closed and order convex set. Suppose $f: U \mapsto U$ is strongly monotone, strictly subhomogeneous and admits a nonempty compact invariant set $K \subset \text{int}(P)$. Then $f$ has a fixed point $e \gg 0$ such that every nonempty compact invariant set of $f$ in $\text{int}(P)$ consists of $e$. 
\end{lemma}

\begin{lemma}\label{Theorem 3.4.8 [H88]}
(Theorem 3.4.8, \cite{H88}) If there exists $t_{1} \geq 0$ such that the $C^{r}$-semigroup $T(t): Y \mapsto Y, t \geq 0$, $Y$ any metric space, is completely continuous for $t > t_{1}$ and point dissipative, then there exists a global attractor $A$. If $Y$ is a Banach space, then $A$ is connected and if $t_{1} = 0$, then there is an equilibrium point of $T(t)$. 
\end{lemma}

\begin{lemma}\label{Lemma 3, [SZ01]}
(Lemma 3, \cite{SZ01}) Let $Y$ be a metric space, $\Psi$ a semiflow on $Y$, $Y_{0} \subset Y$ an open set, $\partial Y_{0} = Y \setminus Y_{0}$, $M_{\partial} = \{y \in \partial Y_{0}: \Psi_{t}(y) \in \partial Y_{0}  \hspace{1mm} \forall  \hspace{1mm} t \geq 0 \}$ and $q$ be a generalized distance function for semiflow $\Psi$. Assume that 
\begin{enumerate}
\item $\Psi$ has a global attractor $A$, 
\item there exists a finite sequence $K = \{K_{i}\}_{i=1}^{n}$ of pairwise disjoint, compact and isolated invariant sets in $\partial Y_{0}$ with the following properties
\begin{itemize}
\item $\cup_{y \in M_{\partial}} \omega(y) \subset \cup_{i=1}^{n} K_{i}$,
\item no subset of $K$ forms a cycle in $\partial Y_{0}$, 
\item $K_{i}$ is isolated in $Y$, 
\item $W^{s}(K_{i}) \cap q^{-1}(0,\infty) = \emptyset  \hspace{1mm} \forall  \hspace{1mm} i = 1,\hdots, n$. 
\end{itemize}
\end{enumerate}
Then there exists $\delta > 0$ such that for any compact chain transitive set $L$ that satisfies $L \not\subset K_{i}  \hspace{1mm} \forall  \hspace{1mm} i = 1, \hdots, n$, $\min_{y \in L} q(y) > \delta$ holds. 
\end{lemma}

\begin{lemma}\label{pg. 3, [Z03]}
(pg. 3, \cite{Z03}) Suppose the Kuratowski's measure of non-compactness for any bounded set $B$ of $Y$, any metric space, is denoted by 
\begin{equation*}
\alpha(B) = \inf\{r: B \text{ has a finite cover of diameter } r\}.
\end{equation*}
Firstly, $\alpha(B) = 0$ if and only if $\overline{B}$ is compact.  

Moreover, a continuous mapping $f: Y \mapsto Y, Y$ any metric space, is $\alpha$-condensing ($\alpha$-contraction of order $0 \leq k < 1$) if $f$ takes bounded sets to bounded sets and $\alpha(f(B)) < \alpha(B)$ ($\alpha(f(B)) \leq k \alpha (B)$) for any nonempty closed bounded set $B \subset Y$ such that $\alpha(B) > 0$. Moreover, $f$ is asymptotically smooth if for any nonempty closed bounded set $B \subset Y$ for which $f(B) \subset B$, there exists a compact set $J \subset B$ such that $J$ attracts $B$. 

It is well-known that a compact map is an $\alpha$-contraction of order $0$, and an $\alpha$-contraction or order $k$ is $\alpha$-condensing. Moreover, by Lemma 2.3.5, \cite{H88}, any $\alpha$-condensing maps are asymptotically smooth.  
\end{lemma}

\begin{lemma}\label{Theorem 3.7, [MZ05]}
(Theorem 3.7, \cite{MZ05}) Let $(M, d)$ be a complete metric space, and $\rho: M \to [0, \infty)$ a continuous function such that $M_{0} = \{x \in M: \rho(x) > 0\}$ is nonempty and convex. Suppose that $T: M \mapsto M$ is continuous, asymptotically smooth, $\rho$-uniformly persistent, $T$ has a global attractor $A$ and satisfies $T(M_{0}) \subset M_{0}$. Then $T: (M_{0}, d)  \mapsto (M_{0}, d)$ has a global attractor $A_{0}$. 
\end{lemma}

\begin{remark}\label{Remark 3.10, [MZ05]}
(Remark 3.10, \cite{MZ05}) Let $(M, d)$ be a complete metric space. A family of mappings $\Psi_{t}: M \mapsto M, t \geq 0$, is called a continuous-time semiflow if $(x,t) \mapsto \Psi_{t}(x)$ is continuous, $\Psi_{0} = Id$ and $\Psi_{t}\circ \Psi_{s}  =\Psi_{t+s}$ for $t, s \geq 0$. Lemma \ref{Theorem 3.7, [MZ05]} is valid even if replaced by a continuous-time semiflow $\Psi_{t}$ on $M$ such that $\Psi_{t}(M_{0}) \subset M_{0} \hspace{1mm} \forall  \hspace{1mm} t \geq 0$. 
\end{remark}

\begin{lemma}\label{Theorem 4.7, [MZ05]}
(Theorem 4.7, \cite{MZ05}) Let $M$ be a closed convex subset of a Banach space $(X, \lVert \cdot \rVert)$, $\rho: M \to [0, \infty)$ a continuous function such that $M_{0} = \{x \in M: \rho(x) > 0\}$, where $M_{0}$ is nonempty and convex, and $\Psi_{t}$ a continuous-time semiflow on $M$ such that $\Psi_{t}(M_{0}) \subset M_{0} \hspace{1mm}  \forall  \hspace{1mm} t  \geq 0$. If either $\Psi_{t}$ is $\alpha$-condensing $ \hspace{1mm} \forall  \hspace{1mm} t > 0$ or $\Psi_{t}$ is convex $\alpha$-contracting for $t > 0$, and $\Psi_{t}: M_{0} \mapsto M_{0}$ has a global attractor $A_{0}$, then $\Psi_{t}$ has an equilibrium $a_{0} \in A_{0}$. 

\end{lemma}

\section{Key Proposition}
Many authors found Lemma 1 of \cite{LZ11} to be very useful in various proofs (see e.g. Lemma 3.2, \cite{VWZ12}). The key to the proof of our claim is the following extension of Lemma 1 of \cite{LZ11} to consider the case with convection: 
\begin{proposition}
Consider in a spatial domain with $x \in [0,1]$, the following scalar reaction-convection-diffusion equation
\begin{equation}\label{11}
\begin{cases}
\partial_{t} w(x,t) = \overline{D} \partial_{xx}^{2} w(x,t) - \overline{U} \partial_{x} w(x,t) + g(x) - \lambda w(x,t),\\
\overline{D}\partial_{x} w(x,t) - \overline{U} w(x,t) \rvert_{x=0} = \partial_{x} w(x,t) \rvert_{x=1} = 0, \hspace{3mm} w(x,0) = \psi(x),
\end{cases}
\end{equation}
where $\overline{D} > 0, \lambda > 0$, $\overline{U} \geq 0$, and $g(x) > 0$ is a continuous function. Then $\forall \psi \in C([0,1], \mathbb{R}_{+})$, there exists a unique positive steady state $w^{\ast}$ which is globally attractive in $C([0,1], \mathbb{R})$. Moreover, in the case $\overline{U} = 0$ and $g(x) \equiv g$, it holds that $w^{\ast} = \frac{g}{\lambda}$.
\end{proposition}
\begin{proof}
The case $\overline{U} = 0$ is treated in Lemma 1 of \cite{LZ11}; we assume $\overline{U} > 0$ here. By continuity we know that there exists 
\begin{equation*}
0 < \min_{x \in [0,1]} g(x) \leq g(x) \leq \max_{x \in [0,1]} g(x) \triangleq \overline{g}\hspace{1mm}  \forall  \hspace{1mm} x \in [0,1].
\end{equation*}
We define $F(x,w) \triangleq g(x) - \lambda w(x,t)$. It is immediate that (e.g. by Lemma \ref{Theorem 7.3.1, Corollary 7.3.2, [S95]} and Remark \ref{pg. 121, [S95]}) $ \hspace{1mm} \forall  \hspace{1mm} \psi \in C([0,1], \mathbb{R}_{+})$, there exists a unique solution $w = w(x,t, \psi) \in C([0,1], \mathbb{R}_{+})$ on some time interval $[0, \sigma), \sigma = \sigma(\psi)$. 

We fix $\psi \in C([0,1], \mathbb{R}_{+})$ so that by continuity there exists $\max_{x \in [0,1]} \psi(x)$. Now if $v \equiv M$ for $M$ sufficiently large such that $M > \max\{\max_{x \in [0,1]} \psi(x), \frac{\overline{g}}{\lambda} \}$, then by Theorem 7.3.4 of \cite{S95} and the blow up criterion from Lemma \ref{Theorem 7.3.1, Corollary 7.3.2, [S95]} and Remark \ref{pg. 121, [S95]}, we immediately deduce the existence of a unique solution on $[0, \infty)$.

Hence, there exists  the solution semiflow $P_{t}$ such that $P_{t}(\psi) = w(t,\psi), \psi \in C([0,1], \mathbb{R}_{+})$. It follows that 
\begin{equation*}
\omega(\psi) \subset \{\varphi: \frac{\min_{x \in [0,1]} g(x)}{\lambda} \leq \varphi \leq \frac{\max_{x \in [0,1]} g(x)}{\lambda}\}
\end{equation*}
by comparison principle (e.g. Theorem 7.3.4 \cite{S95}); we emphasize here again that as stated on pg. 121, \cite{S95}, Theorem 7.3.4 \cite{S95} is applicable to the general second-order differentiation operator such as $\overline{D} \partial_{xx}^{2} - \overline{U} \partial_{x}$. By comparison principle again (e.g. Corollary 7.3.5, Theorem 7.4.1, \cite{S95}), it also follows that 
\begin{equation*}
P_{t}(\psi_{1}) \gg P_{t}(\psi_{2}) \hspace{3mm} \forall  \hspace{1mm} t > 0
\end{equation*}
if $\psi_{1} > \psi_{2}$; this implies that $P_{t}$ is strongly monotone (see Definition \ref{pg. 38, 40, 46, [Z03]}). Moreover, $F$ is strictly subhomogeneous (see Definition \ref{pg. 38, 40, 46, [Z03]}) in a sense that $F(x, \alpha w) > \alpha F(x, w)$ $ \hspace{1mm} \forall  \hspace{1mm} \alpha \in (0,1)$ as $g(x) > 0$. We now follow the idea from pg. 348 \cite{FZ97} to complete the proof. Let $L(t) \triangleq w(t, \alpha \psi) - \alpha w(t,\psi)$ so that 
\begin{align*}
&\partial_{t} L = \overline{D} \partial_{xx}^{2} L - \overline{U} \partial_{x} L + (1-\alpha) g(x) - \lambda L,\\
&L(0) = 0, \hspace{3mm} \overline{D}\partial_{x} L - \overline{U}L \rvert_{x=0} = \partial_{x} L \rvert_{x=1} = 0.
\end{align*}
Let $\Psi(t,s), t \geq s \geq 0$ be the evolution operator of 
\begin{equation}\label{12}
\begin{cases}
\partial_{t}N = \overline{D} \partial_{xx}^{2}N - \overline{U} \partial_{x} N - \lambda N,\\
\overline{D} \partial_{x}N - \overline{U} N \rvert_{x=0} = \partial_{x} N \rvert_{x=1} = 0.
\end{cases}
\end{equation}
Then $\Psi(t, 0)(0) = 0$. Thus, by Theorem 7.4.1 \cite{S95}, which is applicable to the general second-order differentiation operator such as $\overline{D}\partial_{xx}^{2} - \overline{U} \partial_{x}$, we see that $ \hspace{1mm} \forall  \hspace{1mm} \psi > 0, \Psi(t,s) \psi \gg 0$. Hence by Comparison Principle as $g(x) (1-\alpha) \geq 0$, we obtain $ \hspace{1mm} \forall  \hspace{1mm} \psi > 0$, $L(x, t, \psi) \gg 0$. Therefore, $ \hspace{1mm} \forall  \hspace{1mm} \psi > 0, w(t,\alpha \psi) > \alpha w(t,\psi)$; i.e. $P_{t}$ is strictly subhomogeneous (see Definition \ref{pg. 38, 40, 46, [Z03]}).  

By Lemma \ref{Theorem 2.3.2, [Z03]} we now conclude that $P_{t}$ has a fixed point $w^{\ast}(x) \gg 0$ such that $\omega(\psi) = w^{\ast} \in C([0,1], \mathbb{R}_{+})  \hspace{1mm} \forall  \hspace{1mm} \psi \in C([0,1], \mathbb{R}_{+})$.   
\end{proof}

\section{Proof of Theorem 2.1}

Firstly, by Lemma \ref{Theorems 2.1, 2.2, [YW16]}, we know that given $\phi \in X^{+}$, there exists a unique global nonnegative solution to the system (\ref{1}) subjected to (\ref{2}), (\ref{3}). 

Now, from the proof of Theorem 2.3 (1) \cite{YW16}, we know that if we linearize (\ref{1}) about the DFE $(S, I, R, B) = (m^{\ast}, 0, 0, 0)$, we obtain 
\begin{equation}\label{13}
\begin{cases}
\partial_{t} S = D \partial_{xx}^{2}S - m^{\ast} \left(\beta_{1} I +  \frac{\beta_{2}}{K} B\right) -d \, S+ \sigma R,\\
\partial_{t} I = D \partial_{xx}^{2}I  + m^{\ast} \left(\beta_{1}I + \frac{\beta_{2}}{K} B\right)-I(d+\gamma),\\
\partial_{t} R = D \partial_{xx}^{2}R +  \gamma I- R(d+\sigma),\\
\partial_{t} B = D_{4} \partial_{xx}^{2}B - U\partial_{x}B + \xi I + gB -\delta B,
\end{cases}
\end{equation}
so that substituting $(S,I,R,B) = (e^{\lambda t} \psi_{1}(x), e^{\lambda t} \psi_{2}(x), e^{\lambda t} \psi_{3}(x), e^{\lambda t} \psi_{4}(x))$ in (\ref{13}) gives us the eigenvalue problem of 
\begin{align}\label{14}
\begin{cases}
\lambda \psi_{1} = D \partial_{xx}^{2} \psi_{1} - m^{\ast} \left(\beta_{1} \psi_{2} + \frac{\beta_{2}}{K} \psi_{4} \right) - d \psi_{1} + \sigma \psi_{3}, \\
\lambda \psi_{2} = D \partial_{xx}^{2} \psi_{2} + m^{\ast} \left( \beta_{1} \psi_{2} + \frac{\beta_{2}}{K} \psi_{4} \right) - \psi_{2}(d+\gamma),\\
\lambda \psi_{3} = D \partial_{xx}^{2} \psi_{3} + \gamma \psi_{2} - \psi_{3}(d+\sigma),\\
\lambda \psi_{4} = D_{4} \partial_{xx}^{2} \psi_{4} - U \partial_{x} \psi_{4} + \xi \psi_{2} + g \psi_{4} - \delta \psi_{4}.
\end{cases}
\end{align}
We define 
\begin{align}\label{15}
\tilde{\Theta} (\psi_{1}, \psi_{2}, \psi_{3}, \psi_{4}) 
\triangleq
\begin{pmatrix}
D \partial_{xx}^{2} \psi_{1} - m^{\ast} \left(\beta_{1} \psi_{2} + \frac{\beta_{2}}{K} \psi_{4} \right) - d \psi_{1} + \sigma \psi_{3}\\
D \partial_{xx}^{2} \psi_{2} + m^{\ast} \left( \beta_{1} \psi_{2} + \frac{\beta_{2}}{K} \psi_{4} \right) - \psi_{2}(d+\gamma)\\
D \partial_{xx}^{2} \psi_{3} + \gamma \psi_{2} - \psi_{3}(d+\sigma) \\
D_{4} \partial_{xx}^{2} \psi_{4} - U \partial_{x} \psi_{4} + \xi \psi_{2} + g \psi_{4} - \delta \psi_{4}
\end{pmatrix}.
\end{align}
It is shown in the proof of Theorem 2.3 (1) \cite{YW16} that defining 
\begin{align}\label{16}
\Theta 
\begin{pmatrix}
\psi_{2}\\
\psi_{4}
\end{pmatrix}
\triangleq& 
\left( 
\begin{pmatrix}
D \partial_{xx}^{2} - (d+\gamma) & 0 \\
0 & D_{4} \partial_{xx}^{2} - U \partial_{x} - \delta 
\end{pmatrix}
+ 
\begin{pmatrix}
m^{\ast} \beta_{1} & m^{\ast} \frac{\beta_{2}}{K}\\
\xi & g 
\end{pmatrix}
\right)
\begin{pmatrix}
\psi_{2}\\
\psi_{4}
\end{pmatrix}\nonumber\\
=& (\Theta_{2} + \Theta_{1})
\begin{pmatrix}
\psi_{2}\\
\psi_{4}
\end{pmatrix},
\end{align}
we have the spectral bound of $\Theta_{2}$, $s(\Theta_{2})$, to satisfy $s(\Theta_{2}) < 0$. Thus, by Theorem 3.5 \cite{T09}, $s(\Theta)$, the spectral bound of $\Theta$, and hence $s(\tilde{\Theta})$, due to the independence of $\Theta$ from the first and third equations of $\tilde{\Theta}(\psi_{1}, \psi_{2}, \psi_{3}, \psi_{4})$ in (\ref{15}) , has the same sign as 
\begin{equation*}
r(-\Theta_{1} \Theta_{2}^{-1}) - 1 = \mathcal{R}_{0} - 1. 
\end{equation*}
That is, $\mathcal{R}_{0} - 1$ and the principal eigenvalue of $\tilde{\Theta}$, $\lambda = \lambda(m^{\ast})$, have same signs. Now by hypothesis, $\mathcal{R}_{0} < 1$ and hence $\mathcal{R}_{0} - 1 < 0$ so that $\lambda(m^{\ast}) < 0$. This implies 
\begin{equation*}
\lim_{\epsilon \to 0} \lambda (m^{\ast} + \epsilon) = \lambda(m^{\ast}) < 0
\end{equation*}
and therefore, there exists $\epsilon_{0} > 0$ such that $\lambda (m^{\ast} + \epsilon_{0}) < 0$. Let us fix this $\epsilon_{0} > 0$. 

By \cite{YW16} (see (14a), (14b), (14c) of \cite{YW16}), we know that defining $V \triangleq S + I + R$, we obtain 
\begin{equation}\label{17}
\partial_{t} V = D \partial_{xx}^{2} V + b -dV, \hspace{3mm} 
\partial_{x} V \rvert_{x=0,1} = 0, \hspace{3mm} V(x,0) = V_{0}(x)
\end{equation}
where $V_{0}(x) \triangleq \phi_{1}(x) + \phi_{2}(x) + \phi_{3}(x), D > 0, b > 0, d > 0$. By Proposition 4.1 with $\overline{U} = 0, g(x) \equiv b, \lambda = d$, we see that (\ref{17}) admits a unique positive steady state $m^{\ast} = \frac{b}{d}$ which is globally attractive in $C([0,1], \mathbb{R}_{+})$. Therefore, due to the non-negativity of $S, I, R$, for the fixed $\epsilon_{0} > 0$, there exists $t_{0} = t_{0}(\phi)$ such that $ \hspace{1mm} \forall  \hspace{1mm} t \geq t_{0}, x \in [0,1]$, $S(t,x) \leq m^{\ast} + \epsilon_{0}$. Thus, $ \hspace{1mm} \forall \hspace{1mm}  t \geq t_{0}, x \in [0,1]$, 
\begin{align}\label{18}
\partial_{t}I 
\leq D \partial_{xx}^{2} I + \beta_{1} (m^{\ast} + \epsilon_{0}) I + \frac{\beta_{2}B}{K} (m^{\ast} + \epsilon_{0}) - I(d+\gamma) 
\end{align}
by (\ref{1}) as $B \geq 0$ and 
\begin{align}\label{19}
\partial_{t}B 
\leq D_{4} \partial_{xx}^{2}B - U \partial_{x}B + \xi I + B(-\delta) + gB  
\end{align}
by (\ref{1}) as $B^{2}\geq 0, g > 0, K_{B} > 0$. As we will see, it was crucial above how we take these upper bounds carefully. Thus, we now consider for $x \in [0,1], t \geq t_{0}$,  
\begin{equation}\label{20}
\begin{cases}
\partial_{t} V_{2} = D \partial_{xx}^{2}V_{2} + \beta_{1}(m^{\ast} + \epsilon_{0}) V_{2} + \frac{\beta_{2}V_{4}}{K} (m^{\ast} + \epsilon_{0}) - V_{2}(d+\gamma),\\
\partial_{t}V_{4} = D_{4} \partial_{xx}^{2}V_{4} - U \partial_{x}V_{4} + \xi V_{2} + V_{4}(-\delta) + g V_{4},
\end{cases}
\end{equation}
for which its corresponding eigenvalue problem obtained by substituting $(V_{2}, V_{4}) = (e^{\lambda t} \psi_{2}(x), e^{\lambda t} \psi_{4}(x))$ in (\ref{20}) is 
\begin{equation}\label{21}
\begin{cases}
\lambda \psi_{2} = D \partial_{xx}^{2}\psi_{2} + \beta_{1}(m^{\ast} + \epsilon_{0}) \psi_{2} + \frac{\beta_{2}\psi_{4}}{K} (m^{\ast} + \epsilon_{0}) - \psi_{2}(d+\gamma),\\
\lambda \psi_{4} = D_{4} \partial_{xx}^{2}\psi_{4} - U \partial_{x}\psi_{4} + \xi \psi_{2} + \psi_{4}(-\delta) + g \psi_{4}.
\end{cases}
\end{equation} 
We may write this right hand side as 
\begin{align}\label{22}
& \begin{pmatrix}
D \partial_{xx}^{2}\psi_{2}\\
D_{4} \partial_{xx}^{2}\psi_{4} - U \partial_{x}\psi_{4}
\end{pmatrix}
+ 
\begin{pmatrix}
\beta_{1} (m^{\ast} + \epsilon_{0}) - (d+\gamma) & \frac{\beta_{2}}{K} (m^{\ast} + \epsilon_{0}) \\
\xi & g-\delta 
\end{pmatrix}
\begin{pmatrix}
\psi_{2}\\
\psi_{4}
\end{pmatrix}\\
\triangleq& \begin{pmatrix}
D \partial_{xx}^{2}\psi_{2}\\
D_{4} \partial_{xx}^{2}\psi_{4} - U \partial_{x}\psi_{4}
\end{pmatrix}
+ 
M(x) 
\begin{pmatrix}
\psi_{2}\\
\psi_{4}
\end{pmatrix} \nonumber 
\end{align}
so that $M_{ij} \geq 0  \hspace{1mm} \forall  \hspace{1mm} i \neq j$ as $\xi, \frac{\beta_{2}}{K} (m^{\ast} + \epsilon_{0}) > 0$. Moreover, it is also clear that $M$ is irreducible as $M_{12}, M_{21} > 0$ (see Definition \ref{pg. 56, 129 [S95]}). Therefore, by Theorem 7.6.1 \cite{S95}, the eigenvalue problem of (\ref{21}) has a real eigenvalue $\overline{\lambda}$ and its corresponding positive eigenfunction $\psi_{0}$. 

Now we recall that $\lambda(m^{\ast})$ is the principal eigenvalue of (\ref{15}) and make a key observation that the second and fourth equations are independent of the first and third equations and therefore, $\lambda(m^{\ast})$ must also be the eigenvalue of 
\begin{align}\label{23}
&\begin{pmatrix}
D \partial_{xx}^{2} \psi_{2} + m^{\ast} \left( \beta_{1} \psi_{2} + \frac{\beta_{2}}{K} \psi_{4} \right) - \psi_{2}(d+\gamma)\\
D_{4} \partial_{xx}^{2} \psi_{4} - U \partial_{x} \psi_{4} + \xi \psi_{2} + g \psi_{4} - \delta \psi_{4}.
\end{pmatrix}\\
=& 
\begin{pmatrix}
D\partial_{xx}^{2} \psi_{2} \\
D_{4} \partial_{xx}^{2} \psi_{4} - U \partial_{x} \psi_{4}
\end{pmatrix}
+ 
\begin{pmatrix}
m^{\ast} \beta_{1} - (d+\gamma) & m^{\ast} \frac{\beta_{2}}{K}\\
\xi & g - \delta
\end{pmatrix}
\begin{pmatrix}
\psi_{2}\\
\psi_{4}
\end{pmatrix}. \nonumber 
\end{align}
Moreover, we observe that replacing $m^{\ast}$ with $m^{\ast} + \epsilon_{0}$ gives us the eigenvalue problem (\ref{21}). Hence, $\overline{\lambda} = \lambda(m^{\ast} + \epsilon_{0}) < 0$ is the principal eigenvalue of (\ref{21}) which therefore has a solution of 
\begin{equation*}
e^{\lambda (m^{\ast} + \epsilon_{0}) (t - t_{0})} \psi_{0}(x), \hspace{3mm} t \geq t_{0}. 
\end{equation*}
Now we find $\eta > 0$ sufficiently large so that 
\begin{equation*}
(I(x, t_{0}),  B(x, t_{0})) \leq \eta \psi_{0}(x) 
\end{equation*}
which is possible as $\psi_{0}$ is positive. Considering (\ref{9}), we may define 
\begin{subequations}\label{24}
\begin{align}
& F_{2}^{+} \triangleq \beta_{1}(m^{\ast} + \epsilon_{0}) I + \frac{\beta_{2}B}{K} (m^{\ast} + \epsilon_{0}) - I(d+\gamma),\\
& F_{4}^{+} \triangleq \xi I + B(-\delta) + g B, 
\end{align}
\end{subequations}
so that 
\begin{equation*}
\frac{\partial F_{2}^{+}}{\partial B} = \frac{\beta_{2}}{K}(m^{\ast}+ \epsilon_{0}) \geq 0, \hspace{3mm} \frac{\partial F_{4}^{+}}{\partial I} = \xi \geq 0,
\end{equation*}
and hence 
$
\begin{pmatrix}
F_{2}^{+}\\
F_{4}^{+}
\end{pmatrix}$ is cooperative (see Definition \ref{pg. 56, 129 [S95]}). By comparison principle, or specifically Theorem 7.3.4 \cite{S95}, due to (\ref{18}), (\ref{19}), (\ref{24}), we obtain $ \hspace{1mm} \forall  \hspace{1mm} t \geq t_{0}, x \in [0,1]$, 
\begin{align*}
(I(x, t), B(x, t)) \leq \eta e^{\lambda (m^{\ast} + \epsilon_{0}) (t-t_{0})} \psi_{0}(x)
\end{align*}
where $\eta e^{\lambda (m^{\ast} + \epsilon_{0}) (t-t_{0})} \psi_{0}(x)\to 0 $ as $t\to \infty$ because $\lambda(m^{\ast} + \epsilon_{0}) < 0$. 

Thus, the equation for $R$, by (\ref{1}), is asymptotic to 
\begin{equation*}
\partial_{t}V_{3} = D\partial_{xx}^{2}V_{3} - V_{3}(d+\sigma)
\end{equation*}
and hence by the theory of asymptotically autonomous semiflows (see Corollary 4.3 \cite{T92}), we have $\lim_{t\to\infty} R(x, t) = 0$. As we noted already, (\ref{17}) admits a unique positive steady state $m^{\ast}$ which is globally attractive, and we just showed that $ \hspace{1mm} \forall  \hspace{1mm} x \in [0,1], \lim_{t\to\infty} I(x, t) = \lim_{t\to\infty} R(x, t) = 0$, and therefore we obtain $\lim_{t\to\infty} S(x, t) = m^{\ast}$. This completes the proof of Theorem 2.1. 

\section{Proof of Theorem 2.2} 

We need the following proposition:
\begin{proposition}
Let $u(x, t, \phi)$ be the solution of the system (\ref{1}) with $D = D_{1} = D_{2} = D_{3}$, subjected to (\ref{2}), (\ref{3}) such that $u(x, 0, \phi) = \phi \in X^{+}$. If there exists some $t_{0}^{I} \geq 0$ such that $I(\cdot, t_{0}^{I}) \not\equiv 0$, then $I(x, t) > 0  \hspace{1mm} \forall  \hspace{1mm} t > t_{0}^{I}, x \in [0,1]$. Similarly, if there exists some $t_{0}^{R} \geq 0$ such that $R(\cdot, t_{0}^{R}) \not\equiv 0$, then $R(x, t) > 0  \hspace{1mm} \forall  \hspace{1mm} t > t_{0}^{R}, x \in [0,1]$. Finally, if there exists some $t_{0}^{B} \geq 0$ such that $B(\cdot, t_{0}^{B}) \not\equiv 0$, then $B(x, t) > 0  \hspace{1mm} \forall  \hspace{1mm} t > t_{0}^{B}, x \in [0,1]$. 

Moreover, for any $\phi \in X^{+},$ it always holds that $S(x, t) > 0  \hspace{1mm} \forall  \hspace{1mm} x \in [0,1], t > 0$ and 
\begin{equation*}
\liminf_{t\to\infty} S(\cdot, t, \phi) \geq \frac{b}{\beta_{1} 2m^{\ast} + \beta_{2} + d}.
\end{equation*}
\end{proposition}

\begin{proof}
We observe that by (\ref{1}), 
\begin{align}\label{25}
\partial_{t} I 
\geq D \partial_{xx}^{2} I - I(d+\gamma).
\end{align}
\begin{align}\label{26}
\partial_{t} R 
\geq  D \partial_{xx}^{2}R - R(d+\sigma).
\end{align}
Thus, we consider 
\begin{align}\label{27}
\begin{cases}
\partial_{t}V_{2} = D \partial_{xx}^{2}V_{2} - V_{2}(d+\gamma) \triangleq D \partial_{xx}^{2}V_{2}  + \tilde{F}_{2},\\
\partial_{x}V_{2}(x,t) \rvert_{x=0,1} = 0,
\end{cases}
\end{align}
\begin{align}\label{28}
\begin{cases}
\partial_{t}V_{3} = D\partial_{xx}^{2}V_{3} - V_{3}(d+\sigma) \triangleq D\partial_{xx}^{2}V_{3} + \tilde{F}_{3},\\
\partial_{x}V_{3}(x,t) \rvert_{x=0,1} = 0,
\end{cases}
\end{align}
such that $V_{2}(\cdot, t_{0}^{I}) \not\equiv 0, I(\cdot, t_{0}^{I}) \geq V_{2}(\cdot, t_{0}^{I})$, and $V_{3}(\cdot, t_{0}^{R}) \not\equiv 0, R(\cdot, t_{0}^{R}) \geq V_{3}(\cdot, t_{0}^{R})$ respectively. By Lemma \ref{Theorem 7.3.1, Corollary 7.3.2, [S95]}, the solutions to (\ref{27}), (\ref{28}) exist locally in time. For both systems (\ref{27}), (\ref{28}), we may repeat the argument in the proof of Proposition 4.1 for the system (\ref{11}) at $\overline{U} = 0, g(x) \equiv 0, \lambda = d+\gamma, \lambda = d+\sigma$ respectively to obtain the sup-norm bounds of both $V_{2}, V_{3}$; therefore, these solutions exist globally in time by the blowup criterion from Lemma \ref{Theorem 7.3.1, Corollary 7.3.2, [S95]}.

Now since $x \in [0,1]$, a one-dimensional space, we may denote 
\begin{equation*}
LV_{2} \triangleq - D \partial_{xx}^{2}V_{2} + (d+\gamma) V_{2}
\end{equation*}
so that 
\begin{equation*}
\partial_{t}V_{2} + L V_{2} = 0 \text{ in } [0,1] \times (0, T], \hspace{1mm} \forall  \hspace{1mm} T> 0
\end{equation*}
by (\ref{27}). Therefore, if $V_{2}(x^{\ast}, t^{\ast}) = 0$ for some $(x^{\ast}, t^{\ast}) \in (0,1) \times (t_{0}^{I}, T]$, then it has a nonpositive minimum in $[0,1] \times [t_{0}^{I}, T]$ and therefore, $V_{2}$ is a constant on $(0,1) \times (0, t^{\ast}]$ by Maximum Principle (see e.g. Theorem 7.1.12, pg. 367 \cite{E98}). Hence as $V_{2}(x^{\ast}, t^{\ast}) = 0$ for $x^{\ast} \in (0,1)$, we must have $V_{2} (\cdot, \cdot) \equiv 0$ on $(0,1) \times (0, t^{\ast}]$. Since $t^{\ast} \in (t_{0}^{I}, T]$, this implies $V_{2}(\cdot, t_{0}^{I}) \equiv 0$ on $(0,1)$, and hence by continuity in $x$, on $[0,1]$. This is a contradiction to $V_{2}(\cdot, t_{0}^{I}) \not\equiv 0$. 

Therefore, we must have $V_{2}(x,t) > 0 \hspace{1mm} \forall  \hspace{1mm} (x,t) \in (0,1) \times (t_{0}^{I}, T]$ and hence $V_{2}(x,t) > 0  \hspace{1mm} \forall  \hspace{1mm} t > t_{0}^{I}, x \in (0,1)$ due to the arbitrariness of $T > 0$. By Comparison Principle (e.g. Theorem 7.3.4 \cite{S95}), we conclude that due to (\ref{25}),
\begin{equation*}
I(\cdot, t) \geq V_{2}(\cdot, t) > 0 \hspace{3mm} \forall  \hspace{1mm} t > t_{0}^{I}, x \in (0,1). 
\end{equation*}
Making use of the boundary values in (\ref{3}), we conclude that $I(\cdot, t) > 0  \hspace{1mm} \forall  \hspace{1mm} t > t_{0}^{I}, x \in [0,1]$. 

The proof that $R(\cdot, t) > 0  \hspace{1mm} \forall  \hspace{1mm} t > t_{0}^{R}, x \in (0,1)$ is done very similarly. We may denote  
\begin{equation*}
LV_{3} \triangleq - D \partial_{xx}^{2}V_{3} + (d+\sigma) V_{3} 
\end{equation*}
so that 
\begin{equation*}
\partial_{t} V_{3} + LV_{3} = 0 \hspace{3mm} \text{ in } [0,1] \times (0, T] \hspace{3mm} \forall  \hspace{1mm} T > 0
\end{equation*}
by (\ref{28}). An identical argument as in the case of $V_{2}$ using Maximum Principle (e.g. Theorem 7.1.12, pg. 367 \cite{E98}) deduces that $V_{3}(x,t) > 0 \hspace{1mm} \forall  \hspace{1mm} (x,t) \in (0,1) \times (t_{0}^{R}, T]$ and hence $V_{3}(x,t) > 0  \hspace{1mm} \forall  \hspace{1mm} t > t_{0}^{R}, x \in (0,1)$ due to the arbitrariness of $T > 0$. By Comparison Principle (e.g. Theorem 7.3.4 \cite{S95}), we conclude that due to (\ref{26})
\begin{equation*}
R(\cdot, t) \geq V_{3}(\cdot, t) > 0 \hspace{3mm} \forall  \hspace{1mm} t > t_{0}^{R}, x \in (0,1).
\end{equation*}
Relying on the boundary values in (\ref{3}) allows us to conclude that $R(\cdot, t) > 0  \hspace{1mm} \forall  \hspace{1mm} t > t_{0}^{R}, x \in [0,1]$. 

Finally, we fix $t_{0}^{B}$ such that $B(\cdot, t_{0}^{B}) \not\equiv 0$ on $[0,1]$ and then $t> t_{0}^{B}$ arbitrary. We know $B$ exists globally in time due to Lemma \ref{Theorems 2.1, 2.2, [YW16]} and thus fix $T > t_{0}^{B}$ so that $t \in [0, T]$. Then by continuity of $B$ in $(x,t) \in [0,1] \times [0, T],$ there exists $M \triangleq \max_{(x,t) \in [0,1] \times [0, T]} B(x,t)$. 

Now $\forall  \hspace{1mm} (x,t) \in [0,1] \times [0, T]$, 
\begin{align}\label{29}
\partial_{t}B \geq D_{4} \partial_{xx}^{2}B - U \partial_{x}B + (g-\delta) B - \frac{gMB}{K_{B}} 
\end{align}
by (\ref{1}). Thus, we consider 
\begin{equation}\label{30}
\begin{cases}
\partial_{t}V_{4} = D_{4} \partial_{xx}^{2}V_{4}  -U \partial_{x}V_{4} + (g - \delta - \frac{gM}{K_{B}}) V_{4} \triangleq D_{4} \partial_{xx}^{2} V_{4} - U \partial_{x}V_{4} + \tilde{F}_{4},\\
D_{4} \partial_{x} V_{4}(x,t) - U V_{4}(x,t) \rvert_{x=0} = \partial_{x}V_{4}(x,t) \rvert_{x=1} = 0, 
\end{cases}
\end{equation}
such that $V_{4}(\cdot, t_{0}^{B}) \not\equiv 0, B_{4}(\cdot, t_{0}^{B}) \geq V_{4}(\cdot, t_{0}^{B})$.

It follows that the solution $V_{4}$ exists locally in time by Lemma \ref{Theorem 7.3.1, Corollary 7.3.2, [S95]}, Remark \ref{pg. 121, [S95]}. Again, repeating the argument in the proof of Proposition 4.1 for the system (\ref{11}) at $\overline{U} = U, g(x) \equiv 0, \lambda = \frac{gM}{K_{B}} + \delta - g > 0$ due to the hypothesis that $g < \delta$ leads to the sup-norm bound so that the solution exists globally in time by the blowup criterion of Lemma \ref{Theorem 7.3.1, Corollary 7.3.2, [S95]}. Now we may denote 
\begin{equation*}
LV_{4} \triangleq -D_{4} \partial_{xx}^{2} V_{4} + U \partial_{x} V_{4} + \left(\frac{gM}{K_{B}} + \delta - g\right) V_{4} 
\end{equation*}
where $\frac{gM}{K_{B}} + \delta - g \geq \delta - g > 0$ by the hypothesis so that 
\begin{equation*}
\partial_{t}V_{4} + LV_{4} = 0 \hspace{3mm} \text{ in } [0,1] \times (0, T].
\end{equation*}
Therefore, if $V_{4}(x^{\ast}, t^{\ast}) = 0$ for some $(x^{\ast}, t^{\ast}) \in (0,1) \times (t_{0}^{B}, T]$, then it has a nonpositive minimum in $[0,1] \times [t_{0}^{B}, T]$ and hence $V_{4}$ is a constant on $(0,1) \times (0, t^{\ast}]$ by Maximum Principle (e.g. Theorem 7.1.12, pg. 367, \cite{E98}). Hence, as $V_{4}(x^{\ast}, t^{\ast}) = 0$ for $x^{\ast} \in (0,1)$, we must have $V_{4}(\cdot, \cdot) \equiv 0$ on $(0,1) \times (0, t^{\ast}]$. Since $t^{\ast} \in (t_{0}^{B}, T]$, this implies that $V_{4}(\cdot, t_{0}^{B}) \equiv 0$ on $(0,1)$ and hence by continuity in $x$, on $[0,1]$. But this contradicts that $V_{4}(\cdot, t_{0}^{B}) \not\equiv 0$. 

Therefore, we must have $V_{4}(x,t) > 0 \hspace{1mm} \forall \hspace{1mm}  (x,t) \in (0,1) \times (t_{0}^{B}, T]$. By Comparison Principle (e.g. Theorem 7.3.4 \cite{S95}), we conclude that due to (\ref{29})
\begin{equation*}
B(\cdot, t) \geq V_{4}(\cdot, t) > 0 \hspace{3mm} \forall  \hspace{1mm} t \in (t_{0}^{B}, T], x \in (0,1).
\end{equation*}
We conclude that by arbitrariness of $T > t_{0}$ and arbitrariness of $t \in [t_{0}^{B}, T]$, this inequality holds for all $t > t_{0}^{B}$. Making use of the boundary values in (\ref{3}) allows us to conclude that $B(\cdot, t) > 0  \hspace{1mm} \forall  \hspace{1mm} t > t_{0}^{B}, x \in [0,1]$. 

Finally, from the proof of Theorem 2.1, specifically due to (\ref{17}) and an application of Proposition 4.1, we know that there exists $t_{1} = t_{1}(\phi)$ such that $\forall  \hspace{1mm} x \in [0,1], t \geq t_{1}$, $I(x,t,\phi) \leq 2m^{\ast}$. Thus, from (\ref{1}) $\forall  \hspace{1mm} x \in [0,1], t \geq t_{1}$, 
\begin{align}\label{31}
\partial_{t}S 
\geq D \partial_{xx}^{2} S + b - S (\beta_{1} 2m^{\ast} + \beta_{2} + d).
\end{align}	
Hence, we consider
\begin{equation}\label{32}
\begin{cases}
\partial_{t}V_{1} = D \partial_{xx}^{2} V_{1} + b - V_{1} (\beta_{1} 2m^{\ast} + \beta_{2} + d) \triangleq D \partial_{xx}^{2}V_{1} + \tilde{F}_{1},\\
\partial_{x}V_{1}(x,t) \rvert_{x=0,1} = 0.
\end{cases}
\end{equation}	
Firstly, by Lemma \ref{Theorem 7.3.1, Corollary 7.3.2, [S95]}, the existence of the unique nonnegative local solution follows. Again, repeating the argument in the proof of Proposition 4.1 for the system (\ref{11}) at $\overline{U} = 0, g(x) \equiv 0, \lambda = \beta_{1}2m^{\ast} + \beta_{2} + d$ leads to the sup-norm bound so that the global existence of the solution follows due to the blowup criterion of Lemma \ref{Theorem 7.3.1, Corollary 7.3.2, [S95]}. Now we may denote by 
\begin{equation*}
LV_{1} \triangleq - D \partial_{xx}^{2}V_{1} + (\beta_{1} 2m^{\ast} + \beta_{2} +d) V_{1}
\end{equation*}
so that $\partial_{t}V_{1} + LV_{1} = b \geq 0$. Therefore, if $V_{1}(x^{\ast}, t^{\ast}) = 0$ for some $(x^{\ast}, t^{\ast}) \in (0,1) \times (0, T]$ for any $T> 0$, then $V_{1}$ attains a nonpositive minimum over $[0,1] \times [0, T]$ at $(x^{\ast}, t^{\ast}) \in (0,1) \times (0, T]$, then by Maximum Principle (e.g. Theorem 7.1.12, pg. 367, \cite{E98}), $V_{1} \equiv c$ on $(0,1) \times (0, t^{\ast}]$. Since $V_{1}(x^{\ast}, t^{\ast}) = 0$, this implies $V_{1} \equiv 0$ on $(0,1) \times (0, t^{\ast}]$.  But by (\ref{32}), we see that this implies $ 0 =b$ which is a contradiction because $b > 0$. Therefore, we must have $V_{1}(x,t, \phi) > 0  \hspace{1mm} \forall  \hspace{1mm} x \in [0,1], t \in [0,T]$ and hence by the arbitrariness of $T > 0$, $ \hspace{1mm} \forall  \hspace{1mm} t > 0$. By (\ref{31}) and Comparison Principle (e.g. Theorem 7.3.4 \cite{S95}), we conclude that $ \hspace{1mm} \forall  \hspace{1mm} t > 0, x \in [0,1]$, 
\begin{equation*}
S(x,t,\phi) \geq V_{1}(x,t,\phi) > 0.
\end{equation*}
Finally, since (\ref{32}) has a unique positive steady state of $\frac{b}{\beta_{1} 2m^{\ast} + \beta_{2} + d}$ by Proposition 4.1 with $\overline{U} = 0, g(x) \equiv b, \lambda = \beta_{1} 2 m^{\ast} + \beta_{2} + d$, we obtain 
\begin{equation*}
\liminf_{t\to\infty} S(\cdot, t, \phi) \geq \frac{b}{\beta_{1} 2m^{\ast} + \beta_{2} + d}.
\end{equation*}	
\end{proof}

We also need the following proposition: 

\begin{proposition}
Suppose $D = D_{1} = D_{2} = D_{3}, \phi \in X^{+}$ and $g < \delta$. Then the system (\ref{1}) subjected to (\ref{2}), (\ref{3}) admits a unique nonnegative solution $u(x, t,\phi)$ on $[0,1]\times [0,\infty)$, and its solution semiflow $\Phi_{t}: X^{+} \mapsto X^{+}$ has a global compact attractor A. 
\end{proposition}

\begin{proof}
Firstly, by Lemma \ref{Theorems 2.1, 2.2, [YW16]}, the unique nonnegative solution $u(t,\phi)$ exists on $[0,\infty)$. As already used in the proof of Theorem 2.1, we know that (\ref{17}) admits a unique positive steady state $m^{\ast} = \frac{b}{d}$. This implies that, as $S, I, R \geq 0$, there exists $t_{1} > 0$ such that $\sf[ \hspace{1mm} \forall  \hspace{1mm} t \geq t_{1}, S(t), I(t), R(t) \leq 2 m^{\ast}]$. Therefore, $ \hspace{1mm} \forall  \hspace{1mm} t \geq t_{1}$, 
\begin{align*}
\partial_{t} B 
\leq D_{4} \partial_{xx}^{2}B - U \partial_{x}B + \xi 2 m^{\ast} + (g-\delta) B 
\end{align*}
by (\ref{1}). Thus, by Proposition 4.1 with $\overline{U} > 0, g(x) = \xi 2m^{\ast} + x, \lambda = \delta - g$, we see that there exists $t_{2} = t_{2}(\phi) > 0$ large so that $B(t,\phi) \leq \frac{\xi 4 m^{\ast} + 1}{\delta - g}$; here we used the hypothesis that $g < \delta$. Hence, the solution semiflow $\Phi_{t}$ is point dissipative (see Definition \ref{pg. 2, 3, 11 [Z03]}). 

As noted in the Preliminaries section, $T$ is compact. From the definitions of (\ref{10}), it is clear that $F = (F_{1}, F_{2}, F_{3}, F_{4})$ is continuously differentiable and therefore locally Lipschitz in $C([0,T], X^{+})$. Moreover, our diffusion operators including the convection operator $T(t)$ is analytic (see the Preliminaries Section) and thus strongly continuous. It follows that the solution semiflow $\Phi_{t}: X^{+} \mapsto X^{+}$ is compact $ \hspace{1mm} \forall  \hspace{1mm} t > 0$. Therefore, by Lemma \ref{Theorem 3.4.8 [H88]}, we may conclude that $\Phi_{t}$ has a global compact attractor. 
\end{proof}

Now we let 
\begin{equation*}
\mathbb{W}_{0} \triangleq \{\psi = (\psi_{1}, \psi_{2}, \psi_{3}, \psi_{4}) \in X^{+}: \psi_{2}(\cdot) \not\equiv 0 \text{ or } \psi_{4}(\cdot) \not\equiv 0 \}
\end{equation*}
and observe that $\mathbb{W}_{0} \subset X^{+}$ is an open set. Moreover, we define 
\begin{align*}
\partial \mathbb{W}_{0} \triangleq& X^{+} \setminus \mathbb{W}_{0}\\
=& \{\psi = (\psi_{1}, \psi_{2}, \psi_{3}, \psi_{4}) \in X^{+}: \psi_{2} (\cdot) \equiv 0 \text{ and } \psi_{4}(\cdot) \equiv 0 \}.
\end{align*}
By Proposition 6.1, it follows that $\Phi_{t}(\mathbb{W}_{0}) \subset \mathbb{W}_{0}  \hspace{1mm} \forall  \hspace{1mm} t \geq 0$ because if $\psi \in \mathbb{W}_{0}$ is such that $\psi_{2}(\cdot) \not\equiv 0$, then by Proposition 6.1, $I(x, t, \psi) > 0 \hspace{1mm} \forall \hspace{1mm} x \in [0,1], t > 0$ whereas if $\psi \in \mathbb{W}_{0}$ is such that $\psi_{2} \equiv 0$, then by the definition of $\mathbb{W}_{0}$ we must have $\psi_{4}(\cdot) \not\equiv 0$ so that by Proposition 6.1, $B(x, t, \psi) > 0 \hspace{1mm} \forall \hspace{1mm} x \in [0,1], t > 0$.

We now define 
\begin{equation*}
M_{\partial} \triangleq \{ \psi \in \partial \mathbb{W}_{0}: \Phi_{t}(\psi) \in \partial \mathbb{W}_{0} \hspace{1mm} \forall  \hspace{1mm} t \geq 0 \}
\end{equation*}
and let $\omega(\phi)$ be the $\omega$-limit set of the orbit $\gamma^{+} (\phi) \triangleq \{\Phi_{t} (\phi)\}_{t \geq 0}$. 

\begin{proposition}
Suppose $D = D_{1} = D_{2} = D_{3}$ and $g < \delta$. For any $\phi \in X^{+}$, let $u(x,t,\phi)$ be the unique nonnegative solution to the system (\ref{1}) subjected to (\ref{2}), (\ref{3}) such that $u(x,0,\phi) = \phi$. Then $\forall  \hspace{1mm} \psi \in M_{\partial}, \omega(\psi) = \{(m^{\ast}, 0, 0, 0)\}$. 
\end{proposition}

\begin{proof}
We fix $\psi \in M_{\partial}$ so that by definition of $M_{\partial}$, we have $\Phi_{t}(\psi) \in \partial \mathbb{W}_{0}  \hspace{1mm} \forall  \hspace{1mm} t \geq 0$; i.e. 
\begin{equation*}
I(\cdot, t) \equiv 0 \text{ and } B(\cdot, t) \equiv 0 \hspace{3mm} \text{ on } [0,1], \hspace{1mm}  \forall  \hspace{1mm} t \geq 0. 
\end{equation*}
Then $S, R$-equations in (\ref{1}) reduce to 
\begin{align*}
\partial_{t}S =& D \partial_{xx}^{2}S + b - dS + \sigma R,\\
\partial_{t}R =& D \partial_{xx}^{2}R - R (d+\sigma), 
\end{align*}
which leads to $\forall  \hspace{1mm} x \in [0,1]$, 
\begin{equation*}
\lim_{t\to\infty} R(x, t, \psi) = 0. 
\end{equation*}
Hence, the $S$-equation in (\ref{1}) is asymptotic to 
\begin{equation*}
\partial_{t}V_{1} = D \partial_{xx}^{2} V_{1} + b - dV_{1} 
\end{equation*}
and therefore by Proposition 4.1 with $\overline{U} = 0, g(x) \equiv b, \lambda = d$, we obtain 
\begin{equation*}
\lim_{t\to\infty} S(x, t, \psi) = \frac{b}{d} = m^{\ast}  \hspace{1mm} \forall  \hspace{1mm} x \in [0,1]. 
\end{equation*}
\end{proof}

Next, we show that $(m^{\ast}, 0, 0, 0)$ is a weak repeller for $\mathbb{W}_{0}$: 

\begin{proposition}
Suppose $D = D_{1} = D_{2} = D_{3}, \phi \in \mathbb{W}_{0}$ and $g < \delta$. Let $u(x,t,\phi)$ be the unique global nonnegative solution of the system (\ref{1}) subjected to (\ref{2}), (\ref{3}) such that $u(x, 0, \phi) = \phi(x)$ and $\Phi_{t}(\phi) = u(t, \phi)$ be its solution semiflow. If $\mathcal{R}_{0} > 1$, then there exists $\delta_{0} > 0$ such that
\begin{align}\label{33}
\limsup_{t\to \infty} \lVert \Phi_{t}(\phi) - (m^{\ast}, 0, 0, 0)\rVert_{C([0,1])} \geq \delta_{0}. 
\end{align}
\end{proposition}

\begin{proof}
By hypothesis $\mathcal{R}_{0} > 1$ so that $\mathcal{R}_{0} - 1  >0$ and as discussed in the proof of Theorem 2.1, we have $\lambda(m^{\ast}) > 0$ where $\lambda(m^{\ast})$ is the principal eigenvalue of $\tilde{\Theta}$ in (\ref{15}). To reach a contradiction, suppose that there exists some $\psi_{0} \in \mathbb{W}_{0}$ such that $\forall  \hspace{1mm} \delta_{0} > 0$ and hence in particular for $\delta_{0} \in (0, m^{\ast})$, 
\begin{equation}\label{34}
\limsup_{t\to\infty} \lVert \Phi_{t}(\psi_{0}) - (m^{\ast}, 0, 0, 0)\rVert_{C([0,1])} < \delta_{0}.
\end{equation}
This implies that there exists $t_{1} > 0$ sufficiently large such that in particular 
\begin{equation*}
m^{\ast} - \delta_{0} < S(x, t), \hspace{3mm} B(x, t) < \delta_{0} \hspace{3mm} \forall \hspace{1mm} t \geq t_{1}, x \in [0,1], 
\end{equation*}as $\Phi_{t}(\psi_{0}) = (S, I, R, B)(t)$. Thus, we see that due to (\ref{1}), 
\begin{align}\label{35}
\partial_{t}I 
\geq& D \partial_{xx}^{2} I + \beta_{1} (m^{\ast} - \delta_{0})I + (m^{\ast} - \delta_{0}) \frac{\beta_{2}}{(\delta_{0} + K)} B - I(d+\gamma),
\end{align}
\begin{align}\label{36}
\partial_{t}B 
\geq& D_{4} \partial_{xx}^{2}B - U \partial_{x}B + \xi I + gB \left(1 - \frac{\delta_{0}}{K_{B}}\right) - \delta B 
\end{align}
$\forall  \hspace{1mm} t \geq t_{1}, x \in [0,1]$. We thus consider for $t \geq t_{1}, x \in [0,1]$, 
\begin{equation}\label{37}
\begin{cases}
\partial_{t}V_{2} = D \partial_{xx}^{2}V_{2} +\beta_{1}(m^{\ast} - \delta_{0}) V_{2} + (m^{\ast} - \delta_{0}) \frac{\beta_{2}}{(\delta_{0} + K)} V_{4} - V_{2}(d+\gamma),\\
\partial_{t}V_{4} = D_{4} \partial_{xx}^{2} V_{4} - U \partial_{x}V_{4} + \xi V_{2} + gV_{4}\left(1 - \frac{\delta_{0}}{K_{B}}\right) - \delta V_{4}.
\end{cases}
\end{equation}
We may write the right hand side as 
\begin{align}\label{38}
\begin{pmatrix}
D \partial_{xx}^{2}V_{2} \\
D_{4} \partial_{xx}^{2}V_{4} - U \partial_{x}V_{4} 
\end{pmatrix}
+ 
M
\begin{pmatrix}
V_{2}\\
V_{4}
\end{pmatrix}
\end{align}
where 
\begin{equation*}
M \triangleq \begin{pmatrix}
\beta_{1}(m^{\ast} - \delta_{0}) - (d+\gamma) & (m^{\ast} - \delta_{0}) \frac{\beta_{2}}{(\delta_{0} + K)}\\
\xi & g(1- \frac{\delta_{0}}{K_{B}}) - \delta 
\end{pmatrix}
\end{equation*}
and therefore, $M_{ij} \geq 0  \hspace{1mm} \forall  \hspace{1mm} i \neq j$ as $\xi, (m^{\ast} - \delta_{0}) \frac{\beta_{2}}{(\delta_{0} + K)} > 0$ because $\delta_{0} < m^{\ast}$ by assumption. This also implies that it is irreducible as in fact $M_{ij} > 0  \hspace{1mm} \forall  \hspace{1mm} i \neq j$ (see Definition \ref{pg. 56, 129 [S95]}). Therefore, by Theorem 7.6.1 \cite{S95}, we may find a real eigenvalue $\lambda(m^{\ast}, \delta_{0})$ and its corresponding positive eigenfunction $\phi_{0}$ so that this system has a solution 
\begin{equation*}
(V_{2}, V_{4})(x,t) = e^{\lambda(m^{\ast}, \delta_{0})(t - t_{1})} \phi_{0}(x)
\end{equation*}
for $x \in [0,1], t \geq t_{1}$.

Now by assumption, $\psi_{0} \in \mathbb{W}_{0}$ and hence $\psi_{2}(\cdot) \not\equiv 0$ or $\psi_{4}(\cdot) \not\equiv 0$. If $\psi_{2}(\cdot) \not\equiv 0$, then by Proposition 6.1, we know that $I(x, t, \psi_{0}) > 0 \hspace{1mm} \forall \hspace{1mm} x \in [0,1], t > 0$. If for any $t_{0}> 0$, $B(\cdot, t_{0}) \equiv 0 \hspace{1mm} \forall \hspace{1mm} x \in [0,1]$, then by (\ref{1}), $0 = \xi I (x, t_{0})$ which is a contradiction because $\xi > 0$. Therefore, $B(\cdot, t_{0}) \not\equiv 0$ and it follows that by Proposition 6.1, $B(x, t, \psi_{0}) > 0 \hspace{1mm} \forall x \in [0,1], t > t_{0}$ and hence $\forall \hspace{1mm} t > 0$ by arbitrariness of $t_{0} > 0$. 

On the other hand, if $\psi_{4}(\cdot) \not\equiv 0$, then by Proposition 6.1, we know that $B(x, t, \psi_{0}) > 0 \hspace{1mm} \forall \hspace{1mm} x \in [0,1], t > 0$. Now if for any $t_{0} > 0$, $I(\cdot, t_{0}) \equiv 0 \hspace{1mm} \forall \hspace{1mm} x \in [0,1]$, then by (\ref{1}), $0 = \beta_{2} S(x, t_{0})\left( \frac{B(x, t_{0})}{B(x, t_{0})+K}\right)$ which is a contradiction because $\beta_{2} > 0$ and $S(x,t) > 0 \hspace{1mm} \forall x \in [0,1], t > 0$ by Proposition 6.1 as $\psi_{0} \in \mathbb{W}_{0} \subset X^{+}$. Therefore, $I(\cdot, t_{0}) \not\equiv 0$ and it follows that by Proposition 6.1, $I(x, t, \psi_{0}) > 0 \hspace{1mm} \forall x \in [0,1], t > t_{0}$ and hence $\forall \hspace{1mm} t > 0$ by arbitrariness of $t_{0} > 0$. Thus, we conclude that $\forall \hspace{1mm} \psi_{0} \in \mathbb{W}_{0}$, $I(x, t, \psi_{0}) > 0, B(x, t, \psi_{0}) > 0 \hspace{1mm} \forall \hspace{1mm} x \in [0,1], t > 0$ and hence in particular $\forall \hspace{1mm} t \geq t_{1}$.

Hence, we may obtain 
\begin{equation}\label{39}
(I(x,t_{1}, \psi_{0}), B(x, t_{1}, \psi_{0})) \geq \eta \phi_{0}(x)
\end{equation}
for $\eta > 0$ sufficiently small. Therefore, by Comparison Principle, specifically Theorem 7.3.4 \cite{S95} with (\ref{9}), 
\begin{align*}
& F_{2}^{-} \triangleq \beta_{1}(m^{\ast} - \delta_{0}) I + (m^{\ast} - \delta_{0}) \frac{\beta_{2}}{(\delta_{0} + K)} B - I(d+\gamma),\\
& F_{4}^{-} \triangleq \xi I + g B\left(1 -\frac{\delta_{0}}{K_{B}} \right)- \delta B,
\end{align*}
so that
\begin{equation*}
\frac{\partial F_{2}^{-}}{\partial B} = (m^{\ast} - \delta_{0}) \frac{\beta_{2}}{(\delta_{0} + K)} \geq 0, \hspace{3mm} \frac{\partial F_{4}^{-}}{\partial I} = \xi > 0,
\end{equation*}
we obtain for $t \geq t_{1}, x \in [0,1]$, 
\begin{equation*}
(I(x, t, \psi_{0}), B(x, t, \psi_{0})) \geq (V_{2}(x, t, \eta \phi_{0}), V_{4}(x, t, \eta \phi_{0})) = \eta e^{\lambda(m^{\ast}, \delta_{0})(t - t_{1})} \phi_{0}(x)
\end{equation*}
due to linearity of (\ref{37}). Now $\lambda(m^{\ast}) > 0$ and in comparison of the second and fourth equations of (\ref{15}) and (\ref{37}), we see that $\lim_{\delta_{0} \to 0} \lambda(m^{\ast}, \delta_{0}) = \lambda (m^{\ast}) > 0$ so that taking $\delta_{0} \in (0, m^{\ast})$ even smaller if necessary, we have $\lambda(m^{\ast}, \delta_{0}) > 0$.  

Thus, we see that $\eta e^{\lambda (m^{\ast}, \delta_{0})(t-t_{1})} \phi_{0}(x) \to \infty$ as $t \to \infty$ because $\phi_{0}(x) \gg 0$ and $\eta > 0$. This implies $(I, B) (x, t, \psi_{0})$ and hence $(S, I, R, B)(x, t, \psi_{0})$ is unbounded, contradicting 
\small 
\begin{align*}
\limsup_{t\to\infty} (\lVert S(t) - m^{\ast} \rVert_{C([0,1])} + \lVert I(t) \rVert_{C([0,1])} + \lVert R(t) \rVert_{C([0,1])} + \lVert B(t) \rVert_{C([0,1])} ) < \delta_{0} 
\end{align*}
\normalsize
by (\ref{6}) and (\ref{34}). 
Therefore, we have shown that for $\delta_{0} \in (0, m^{\ast})$ sufficiently small, (\ref{33}) holds. 
\end{proof}

Now we define a function $p: X^{+} \mapsto \mathbb{R}_{+}$ by 
\begin{equation*}
p(\psi) \triangleq \min \{ \min_{x \in [0,1]} \psi_{2}(x), \min_{x \in [0,1]} \psi_{4}(x)\}
\end{equation*}
It immediately follows that $p^{-1}((0,\infty)) \subset \mathbb{W}_{0}$.

Now suppose $p(\psi) = 0$ and $\psi \in \mathbb{W}_{0}$. The hypothesis that $\psi \in \mathbb{W}_{0}$ implies that 
\begin{equation*}
\psi_{2}(\cdot) \not\equiv 0 \text{ or } \psi_{4}(\cdot) \not\equiv 0. 
\end{equation*}
This deduces that by the argument in the proof of Proposition 6.4, $I(x, t, \psi) > 0$ and $B(x, t, \psi) > 0$ $\forall \hspace{1mm} t > 0, x \in [0,1]$.  Thus, in this case we deduce that 
\begin{equation*}
\min\{ \min_{x \in [0,1]} I(x, t, \psi), \min_{x \in [0,1]} B(x, t, \psi) \} > 0 \hspace{3mm} \forall \hspace{1mm} t> 0
\end{equation*}
which implies that $p(\Phi_{t}(\psi)) > 0 \hspace{1mm} \forall  \hspace{1mm} t > 0$. 

Next, suppose $p(\psi) > 0$ so that $\psi_{2}(\cdot) \not\equiv 0$ and $\psi_{4}(\cdot) \not\equiv 0$. Thus, by Proposition 6.1, this implies $p(\Phi_{t}(\psi)) > 0 \hspace{1mm} \forall  \hspace{1mm} t > 0$. Hence, we have shown that $p$ is a generalized distance function for the semiflow $\Phi_{t}: X^{+} \mapsto X^{+}$.

We already showed that any forward orbit of $\Phi_{t}$ in $M_{\partial}$ converges to $(m^{\ast}, 0, 0, 0)$ due to Proposition 6.3. Thus, as $\Phi_{t}((m^{\ast}, 0, 0, 0)) = (m^{\ast}, 0, 0, 0)$, $ \{(m^{\ast}, 0, 0, 0)\}$ is a nonempty invariant set that is also a maximal invariant set in some neighborhood of itself and hence by Definition \ref{pg. 2, 3, 11 [Z03]}, it is also isolated. Thus,   if we denote the stable set of $(m^{\ast}, 0, 0, 0)$ by $W^{s}((m^{\ast}, 0, 0, 0))$, we see that $W^{s}((m^{\ast}, 0, 0, 0)) \cap \mathbb{W}_{0} = \emptyset$ as $\mathbb{W}_{0} = \{\psi \in X^{+}: \psi_{2}(\cdot) \not\equiv 0 \text{ or } \psi_{4}(\cdot) \not\equiv 0 \}$. Therefore, making use of Propositions 6.2 and 6.3, we may apply Lemma \ref{Lemma 3, [SZ01]} to conclude that there exists $\eta > 0$ that satisfies
\begin{equation*}
\min_{\psi \in \omega(\phi)} p(\psi) > \eta \hspace{3mm} \forall  \hspace{1mm} \phi \in \mathbb{W}_{0};
\end{equation*}
hence, $\forall  \hspace{1mm} i = 2, 4,$ and $\forall  \hspace{1mm} x \in [0,1]$, 
\begin{equation*}
\liminf_{t\to\infty} u_{i}(x,t, \phi) \geq \eta \hspace{3mm} \forall  \hspace{1mm} \phi \in \mathbb{W}_{0}
\end{equation*}
by (\ref{4}). By taking $\eta$ even smaller if necessary to satisfy $\eta \in (0, \frac{b}{\beta_{1} 2m^{\ast} + \beta_{2}  +d})$, we obtain (\ref{9}) using Proposition 6.1.

Finally, we know as shown in the proof of Proposition 6.2, that $\Phi_{t}$ is compact so that it is asymptotically smooth by Lemma \ref{pg. 3, [Z03]}. Moreover, as we already showed that $\Phi_{t}(\mathbb{W}_{0}) \subset \mathbb{W}_{0}$, by Proposition 6.4, we see that $\Phi_{t}$ is $\rho$-uniformly persistent. We also know due to Proposition 6.2 that $\Phi_{t}: X^{+} \mapsto X^{+}$ has a global attractor $A$. Thus, by Lemma \ref{Theorem 3.7, [MZ05]}, Remark \ref{Remark 3.10, [MZ05]}, $\Phi_{t}: \mathbb{W}_{0} \mapsto \mathbb{W}_{0}$ has a global attractor $A_{0}$. 

This implies that because we already showed that $\Phi_{t}(\mathbb{W}_{0}) \subset \mathbb{W}_{0} \hspace{1mm} \forall  \hspace{1mm} t \geq 0$, $\Phi_{t}$ is compact so that it is $\alpha$-condensing by Lemma \ref{pg. 3, [Z03]}, due to Lemma \ref{Theorem 4.7, [MZ05]}, we see that $\Phi_{t}$ has an equilibrium $a_{0} \in A_{0}$. By Proposition 6.1, it is clear that $a_{0}$ is a positive steady state. This completes the proof of Theorem 2.2. 

\section{Conclusion}
In this article, we have studied a general reaction-diffusion-convection cholera model, which formulates bacterial and human diffusion, bacterial convection, intrinsic pathogen growth and direct/indirect transmission routes.  This general formation of the PDE model allows us to give a thorough investigations on the interactions between the spatial movement of human and bacteria, intrinsic pathogen dynamics and multiple transmission pathways and their contribution of the spatial pattern of cholera epidemics.

The main purpose of this work is to investigate the global dynamics of this PDE model (\ref{1}). To achieve this goal, we have established the threshold results of global dynamics of (\ref{1}) using the basic reproduction number $R_0$. Our analysis shows that if $R_0>1$, the disease will persist uniformly; whereas if $R_0<1$, the disease will die out and the DFE is globally attractive when the diffusion rate of susceptible, infectious and recovered human hosts are identical. These results shed light into the complex interactions of cholera epidemics in terms of model parameters, and their impact on extinction and persistence of the disease. In turn, these findings may suggest efficient implications for the prevention and control of the disease. 

Besides, we would like to mention that there are a number of interesting  directions at this point, that haven't been considered in the present work. One direction is to study seasonal and climatic changes. It is well known that these factors can cause fluctuation of disease contact rates, human activity level, pathogen growth and death rates, etc., which in turn have strong impact on disease dynamics.  The other direction is to model spatial heterogeneity. For instance, taking the diffusion and convection coefficients and other model parameters to be space dependent in 2 dimensional spatial domain (instead of constant values in 1 dimensional region) will better reflect the details of spatial variation. These would make for interesting topics in future investigations. 

\section{Appendix}
\subsection{Proof of Lemma \ref{Theorems 2.1, 2.2, [YW16]}}
In this section, we prove Lemma \ref{Theorems 2.1, 2.2, [YW16]} for completeness. The local existence of unique nonnegative mild solution on $[0, \sigma), \sigma = \sigma(\phi)$, as well as the blow up criterion that if $\sigma = \sigma(\phi) < \infty$, then the sup norm of the solution becomes unbounded as $t$ approaches $\sigma$ from below is shown in the Theorem 2.1 \cite{YW16}. To show that $\sigma = \infty$, we assume that $\sigma < \infty$, fix such $\sigma$ and show the uniform bound which contradicts the blow up criterion. Specifically we show that by performing energy estimates more carefully, keeping track of the dependence on each constant, we may extend Proposition 1 of \cite{YW16} to the case $p  =\infty$. For brevity, we write $L^{p}$ to imply $L^{p}([0,1])$ below for $p \in [1, \infty]$. 

\begin{proposition}
If $u(x, t, \phi) = (S, I, R, B)(x, t, \phi)$ solves (\ref{1}) subjected to (\ref{2}), (\ref{3}) in $[0, \sigma)$, then
\begin{equation*}
\sup_{t \in [0,\sigma)} \lVert u(t) \rVert_{L^{\infty}} \leq 3 (\lVert \phi_{1} \rVert_{L^{\infty}} + \lVert \phi_{2} \rVert_{L^{\infty}} + \lVert \phi_{3} \rVert_{L^{\infty}} + b \sigma)(1 + e^{\sigma g} \xi \sigma) + \lVert \phi_{4} \rVert_{L^{\infty}} e^{\sigma g}
\end{equation*}
\end{proposition}

\begin{proof}
From (\ref{1}), we know from the proof of Proposition 1 \cite{YW16} that defining $V \triangleq S + I + R$, we obtain (\ref{17}). For $p \in [2, \infty)$, it is shown in the proof of Proposition 1 of \cite{YW16} that 
\begin{align*}
\sup_{t\in [0, \sigma)}\lVert V(t) \rVert_{L^{p}} \leq \lVert V_{0} \rVert_{L^{p}} + b \sigma. 
\end{align*}
Now as $S, I, R \geq 0$, 
\begin{align*}
\lVert V \rVert_{L^{p}}^{p} 
\geq& \lVert S \rVert_{L^{p}}^{p} + \lVert I \rVert_{L^{p}}^{p} + \lVert R \rVert_{L^{p}}^{p},\\ 
3 ( \lVert S \rVert_{L^{p}}^{p} + \lVert I \rVert_{L^{p}}^{p} + \lVert R \rVert_{L^{p}}^{p})^{\frac{1}{p}}\geq& \lVert S \rVert_{L^{p}} + \lVert I \rVert_{L^{p}} + \lVert R \rVert_{L^{p}}
\end{align*}	
and hence together, this implies that $\forall  \hspace{1mm} p \in [2, \infty)$ 
\begin{equation*}
\sup_{t \in [0, \sigma)}(\lVert S\rVert_{L^{p}} + \lVert I \rVert_{L^{p}} + \lVert R\rVert_{L^{p}})(t) \leq 3 \sup_{t \in [0, \sigma)} \lVert V(t) \rVert_{L^{p}} \leq 3(\lVert V_{0} \rVert_{L^{p}} + b \sigma ).
\end{equation*}
Taking $p \to \infty$ on the right hand side first and then the left hand side shows that 
\begin{equation}\label{40}
\sup_{t \in [0, \sigma)}( \lVert S \rVert_{L^{\infty}} + \lVert I \rVert_{L^{\infty}} + \lVert R \rVert_{L^{\infty}})(t) \leq 3 (\lVert \phi_{1} \rVert_{L^{\infty}} + \lVert \phi_{2} \rVert_{L^{\infty}} + \lVert \phi_{3} \rVert_{L^{\infty}} + b \sigma)
\end{equation}
due to Minkowski's inequalities and (\ref{2}). Next, a similar procedure shows that, as described in complete in detail in the proof of Proposition 1 of \cite{YW16}, we obtain 
\begin{equation*}
\partial_{t} \lVert B \rVert_{L^{p}} \leq \left( \frac{U^{2}}{4D_{4}(p-1)} + g \right)\lVert B \rVert_{L^{p}} + \xi \lVert I \rVert_{L^{p}}.
\end{equation*}
Thus, Gronwall's inequality type argument shows that via H$\ddot{o}$lder's inequality, 
\begin{align*}
\lVert B(t) \rVert_{L^{p}}
\leq \lVert \phi_{4} \rVert_{L^{\infty}} e^{t\left(\frac{U^{2}}{4D_{4} (p-1)} + g\right)} + \xi \int_{0}^{t} \lVert I(s) \rVert_{L^{\infty}} e^{(t-s)\left(\frac{U^{2}}{4D_{4} (p-1)} + g\right)} ds
\end{align*}	
Now taking $p \to \infty$ on the left hand side and then on the right hand side gives $\forall  \hspace{1mm} t \in [0, \sigma)$ 
\begin{align*}
\lVert B(t) \rVert_{L^{\infty}} \leq \lVert \phi_{4} \rVert_{L^{\infty}} e^{\sigma g} + \xi 3 (\lVert \phi_{1} \rVert_{L^{\infty}} + \lVert \phi_{2} \rVert_{L^{\infty}} + \lVert \phi_{3} \rVert_{L^{\infty}} + b \sigma)e^{\sigma g} \sigma 
\end{align*}	
where we used (\ref{40}). Taking $\sup$ over $t \in [0,\sigma)$ on the left hand side completes the proof. 
\end{proof}
By continuity in space of the local solution in $[0, \sigma)$, the proof of Lemma 3.2 is complete. 

\section{Acknowledgments}
The authors would like to thank anonymous reviewers and the editor for their suggestions that improved this manuscript greatly. This work was partially supported by a grant from the Simons Foundation (\#317047 to Xueying Wang).


\begin{thebibliography}{100}  
\addtolength{\leftmargin}{0.2in} 
\setlength{\itemindent}{-0.2in}  

\bibitem{BCGR09} E. Bertuzzo, R. Casagrandi, M. Gatto, I. Rodriguez-Iturbe, A. Rinaldo, \emph{On spatially explicit models of cholera epidemics}, Journal of the Royal Society Interface (2009), DOI: 10.1098/rsif.2009.0204.

\bibitem{BMRGCBRR11} E. Bertuzzo, L. Mari, L. Righetto, M. Gatto, R. Casagrandi, M. Blokesch, I. Rodriguez-Iturbe, A. Rinaldo, \emph{Prediction of the spatial evolution and effects of
control measures for the unfolding Haiti cholera outbreak}, Geophys. Res. Lett., \textbf{38} (2011), DOI: 10.1029/2011GL046823

\bibitem{CP2} V. Capasso, S. L. Paveri-Fontana, A mathematical model for the 1973 cholera epidemic in the {European Mediterranean} region. \emph{Rev. Epidemiol. Sante}, \textbf{27} (1979), 121--132.

\bibitem{C14} A. Carpenter, \emph{Behavior in the time of cholera: Evidence from the 2008-2009 cholera outbreak in Zimbabwe}, Social Computing, Behavioral-Cultural Modeling and Prediction, Springer (2014), 237--244.

\bibitem{CHL11} D.L. Chao, M.E. Halloran, I.M. Longini Jr., \emph{Vaccination strategies for epidemic cholera in Haiti with implications for the developing world}, Proc. Natl. Acad. Sci. USA \textbf{108} (2011), 7081--7085.

\bibitem{DB11} S.F. Dowell, C.R. Braden, \emph{Implications of the introduction of cholera to Haiti}, Emerg. Infect. Dis., \textbf{17} (2011), 1299--1300. 

\bibitem{ESTD13}  M.C. Eisenberg, Z. Shuai, J.H. Tien, P. van den Driessche, \emph{A cholera model in a patchy environment with water and human movement}, Math. Biosci., \textbf{246} (2013), 105--112.

\bibitem{E98} L. Evans, Partial Differential Equations, American Mathematics Society, Providence, Rhode Island, 1998. 

\bibitem{FZ97} H. I. Freedman, X.-Q. Zhao, \emph{Global asymptotics in some quasimonotone reaction-diffusion systems with delays}, J. Differential Equations, \textbf{137} (1997), 340--362. 

\bibitem{H88} J. K. Hale, \emph{Asymptotic Behavior of Dissipative Systems}, Mathematical surveys and monographs, American Mathematics Society, Providence, Rhode Island, 1988. 

\bibitem{HMS06} D.M. Hartley, J.G. Morris, D.L. Smith, \emph{Hyperinfectivity:
a critical element in the ability of V. cholerae to cause epidemics?} PLoS Med., \textbf{3} (2006), 63--69.  

\bibitem{HWZ13} S.-B. Hsu, F.-B. Wang, X.-Q. Zhao, \emph{Global dynamics of zooplankton and harmful algae in flowing habitats}, J. Differential Equations, \textbf{255} (2013), 265--297.  

\bibitem{LZ11} Y. Lou, X.-Q. Zhao, \emph{A reaction-diffusion malaria model with incubation period in the vector population}, J. Math. Biol., \textbf{62} (2011), 543--568.

\bibitem{MZ05} P. Magal, X.-Q. Zhao, \emph{Global attractors and steady states for uniformly persistent dynamical systems}, SIAM J. Math. Anal., \textbf{37} (2005), 251-275.  

\bibitem{MS90} R. Martin, H. L. Smith, \emph{Abstract functional differential equations and reaction-diffusion systems}, Trans. Amer. Math. Soc., \textbf{321} (1990), 1--44.  

\bibitem{MLWGSM11} Z. Mukandavire, S. Liao, J. Wang, H. Gaff, D.L. Smith,
J.G. Morris, \emph{Estimating the reproductive numbers for the 2008-2009 cholera outbreaks in Zimbabwe}, Proc. Natl. Acad. Sci. USA \textbf{108} (2011), 8767--8772.

\bibitem{NSGFL10} R. L. M. Neilan, E. Schaefer, H. Gaff, K. R. Fister, S. Lenhart, \emph{Modeling optimal intervention strategies for cholera}, B. Math. Biol., \textbf{72} (2010), 2004--2018.

\bibitem{PBFHPGMR11}  R. Piarroux, R. Barrais, B. Faucher, R. Haus, M. Piarroux, J. Gaudart, R. Magloire, D. Raoult, \emph{Understanding the cholera epidemic, Haiti}, Emerg. Infect. Dis., \textbf{17} (2011), 1161--1168.

\bibitem{RBMRBGCMVR12} A. Rinaldo, E. Bertuzzo, L. Mari, L. Righetto, M. Blokesch, M. Gatto, R. Casagrandi, M. Murray, S.M. Vesenbeckh, I. Rodriguez-Iturbe, \emph{Reassessment of the 2010-2011 Haiti cholera outbreak and rainfall-driven multiseason projections}, Proc. Natl. Acad. Sci. USA \textbf{109} (2012), 6602--6607.

\bibitem{SD11} Z. Shuai, P. van den Driessche, \emph{Global dynamics of cholera models with differential infectivity}, Math. Biosci., \textbf{234} (2011), 118--126.

\bibitem{S95} H. L. Smith, \emph{Monotone Dynamical Systems: an Introduction to the Theory of Competitive and Cooperative Systems,} Math. Surveys Monogr. \textbf{41}, American Mathematical Society, Providence, Rhode Island, 1995. 	

\bibitem{SZ01} H. L. Smith, X.-Q. Zhao, \emph{Robust persistence for semidynamical systems}, Nonlinear Anal., \textbf{47} (2001), 6169--6179. 

\bibitem{T92} H. R. Thieme, \emph{Convergence results and a Poincar$\acute{e}$-Bendixson trichotomy for asymptotically autonomous differential equations}, J. Math. Biol., \textbf{30} (1992), 755--763. 

\bibitem{T09} H. R. Thieme, \emph{Spectral bound and reproduction number for infinite-dimensional population structure and time heterogeneity}, SIAM J. Appl. Math., \textbf{70} (2009), 188--211. 

\bibitem{TW11} J.P. Tian, J. Wang, \emph{Global stability for cholera epidemic models}, Math. Biosci., \textbf{232} (2011), 31--41.

\bibitem{TE10} J. H. Tien, D. J. D. Earn, \emph{Multiple transmission pathways and disease dynamics in a waterborne pathogen model}, B. Math. Biol., \textbf{72} (2010), 1506--1533.

\bibitem{TSED15} J.H. Tien, Z. Shuai, M.C. Eisenberg, and P. van den Driessche, \emph{Disease invasion on community net-
works with environmental pathogen movement}, J. Math. Biology, \textbf{70} (2015), 1065--1092.

\bibitem{VWZ12} N. K. Vaidya, F.-B. Wang, X. Zou, \emph{Avian influenza dynamics in wild birds with bird mobility and spatial heterogeneous environment}, Discrete Contin. Dyn. Syst. Ser. B, \textbf{17} (2012), 2829--2848. 

\bibitem{WL12} J. Wang, S. Liao, \emph{A generalized cholera model and epidemic-endemic analysis}, J. Biol. Dyn., \textbf{6} (2012), 568--589.

\bibitem{WM11} J. Wang, C. Modnak, \emph{Modeling cholera dynamics with controls}, Canad. Appl. Math. Quart., \textbf{19} (2011), 255--273.

\bibitem{WPW16} X. Wang, D. Posny, J. Wang, \emph{A Reaction-Convection-Diffusion Model for Cholera Spatial Dynamics}, Discrete Contin. Dyn. Syst. Ser. B, to appear. 

\bibitem{WW15} X. Wang, J. Wang, \emph{Analysis of cholera epidemics with bacterial growth and spatial movement}, J. Biol. Dyn., \textbf{9} (2015), 233--261. 

\bibitem{WZ11} W. Wang, X.-Q. Zhao, \emph{A nonlocal and time-delayed reaction-diffusion model of dengue transmission}, SIAM J. Appl. Math., \textbf{71} (2011), 147--168. 

\bibitem{WZ12} W. Wang, X.-Q. Zhao, \emph{Basic reproduction numbers for reaction-diffusion epidemic models}, SIAM J. Appl. Dyn. Syst., \textbf{11} (2012), 1652--1673. 

\bibitem{WHO2} WHO web page: http://www.who.int/csr/don/$2014{\_}05{\_}30$/en/.

\bibitem{W96} J. Wu, \emph{Theory and Applications of Partial Functional Differential Equations}, Springer, New York, 1996. 

\bibitem{YW16} K. Yamazaki, X. Wang, \emph{Global well-posedness and asymptotic behavior of solutions to a reaction-convection-diffusion cholera epidemic model}, Discrete Contin. Dyn. Syst. Ser. B, \textbf{21} (2016), 1297--1316. 

\bibitem{Z03} X.-Q. Zhao, \emph{Dynamical Systems in Population Biology}, Springer-Verlag, New York, Inc., New York, 2003. 

\end{thebibliography}
\end{document}